\documentclass[11pt,twoside,english]{article}
\usepackage{psfrag,babel}
\usepackage{latexsym}
\usepackage{amsmath,amsthm}
\usepackage{amsfonts}
\usepackage{amsgen}
\usepackage{amssymb}
\usepackage{amstext}
\usepackage{amssymb}
\usepackage{graphicx}
\usepackage{color}
\usepackage{epsfig}
\usepackage[colorinlistoftodos]{todonotes}

\usepackage{tikz}
\usetikzlibrary{decorations.pathreplacing}
\newcommand{\opendot}[2]{\filldraw [fill=white,draw=black, thick] ({#1},{#2}) circle (.13) ;}
\newcommand{\closeddot}[2]{\filldraw [fill=black,draw=black, thick] ({#1},{#2}) circle (.13) ;}

\usepackage{hyperref}
\hypersetup{colorlinks=true, citecolor=blue}

\def\X{\widetilde{X}}
\def\leq{\leqslant}
\def\geq{\geqslant}
\def\N{\mathbb{N}}
\def\R{\mathbb{R}}

\DeclareMathOperator{\diam}{diam}

\newtheorem{Proposition}{Proposition}[section]

\newtheorem{lemma}[Proposition]{Lemma}
\newtheorem{definition}[Proposition]{Definition}

\newtheorem{remark} [Proposition]{Remark}

\newtheorem{theorem}[Proposition]{Theorem}
\newtheorem{corollary}[Proposition]{Corollary}

\graphicspath{{Figures/}}

\newcommand{\simp}{\sim^{\!\scriptscriptstyle+}\!}
\newcommand{\ord}{\preccurlyeq^{\!\scriptscriptstyle+}\!}
\newcommand{\ordm}{\preccurlyeq^{\!\scriptscriptstyle-}\!}

\begin{document}

\title{A spectral decomposition of the attractor of  piecewise  contracting maps of the interval.}

\author{A. Calder\'{o}n$^1$, E.\ Catsigeras$^2$ and P.\ Guiraud$^3$}

\maketitle
\begin{center}
\begin{small}

$^1$ Instituto de Ingenier\'ia Matem\'atica and Centro de Investigaci\'on y Modelamiento de Fen\'omenos Aleatorios Valpara\'iso, Facultad de Ingenier\'ia, Universidad de Valpara\'{\i}so,\\
Valpara\'{\i}so, Chile\\
\texttt{alfredo.calderon@postgrado.uv.cl}

$^2$ Instituto de Matem\'{a}tica y Estad\'istica Rafael Laguardia, Universidad de la Rep\'{u}blica,\\
Montevideo, Uruguay\\
\texttt{eleonora@fing.edu.uy}

$^3$ Instituto de Ingenier\'ia Matem\'atica and Centro de Investigaci\'on y Modelamiento de Fen\'omenos Aleatorios Valpara\'iso, Facultad de Ingenier\'ia, Universidad de Valpara\'{\i}so,\\
Valpara\'{\i}so, Chile\\
\texttt{pierre.guiraud@uv.cl}

\end{small}
\end{center}

\begin{abstract}

We study the asymptotic dynamics of  piecewise contracting maps defined on a compact interval. For maps that are not necessarily injective, but have a finite number of  local extrema and discontinuity points, we prove the existence of a decomposition of the support of the asymptotic dynamics into a finite number of minimal components.  Each component is either a periodic orbit or a minimal Cantor set and such that the $\omega$-limit set of (almost) every point in the interval is exactly one of these  components. Moreover, we show that each component is  the $\omega$-limit set, or the closure of the orbit, of a  one-sided limit of the map at a discontinuity point or at a local extremum.\\


 \end{abstract}

\section{Introduction}\label{Introduction}

%

Let $X\subset\mathbb{R}$ be a compact interval with nonempty interior. A map $f:X\to X$ is a {\em piecewise contracting interval map} (PCIM) if there exist $\lambda\in(0,1)$ and a collection of $N\geq 2$
non-empty disjoint open intervals $X_1,X_2,\ldots,X_N$ such that $X=\bigcup_{i=1}^N\overline{X_i}$ and
\begin{eqnarray}\label{eq1}
|f(x)-f(y)|\leq\lambda\,|x-y|\qquad\forall\,x,y\in X_i,\quad\forall\,i\in\{1,2,\ldots,N\}.
\end{eqnarray}
We call {\em contracting constant} (or {\em contracting rate}) of $f$ the real number $\lambda\in(0,1)$, and {\em contraction pieces}  the elements of the collection $\{X_i\}_{i=1}^N$. 

For a PCIM  $f:X\to X$, we let $c_0,c_N$ denote the extreme points of $X$ and $\Delta:=\{c_1<c_2<\dots<c_{N-1}\}$ denotes the set of the boundaries of the contraction pieces. That is,  $X_1=[c_0,c_1),X_2=(c_1,c_2),\dots,X_N=(c_{N-1},c_N]$. For notational convenience we suppose that $X_1$ and $X_N$ are half-closed, but we may consider also the case where one or both pieces are open by adding $c_0$ and/or $c_N$ to the set $\Delta$. In other words, $\Delta$ must contain all the discontinuity points of the map.

From inequality \eqref{eq1}, it follows that the points of $\Delta$ are removable (maybe continuity points) or jump discontinuities. Therefore, for any $i\in\{1,\dots,N\}$ the map $f|_{X_i}$ admits a unique continuous extension $f_i:\overline{X_i}\to X$, which  besides satisfies \eqref{eq1} for any pair of points in $\overline{X_i}$. The one-sided limits of $f$ at the extreme points  of its contraction pieces  write
\[
d_0:=f_1(c_0),\ d_{N}:=f_N(c_{N}),\ d_i^{-}:=f_i(c_i)\;\text{ and }\; d_i^+:=f_{i+1}(c_i) \quad \forall\,i\in\{1,\dots,N-1\}.
\]
We let $D$ denote the set $\{d_0,d_1^-,\ldots,d_{N-1}^-,d_1^+,\ldots,d_{N-1}^+,d_N\}$.

In this paper, our purpose is to describe the topological structure and dynamical properties of the
asymptotic dynamics of PCIM. To this aim, let $f$ be a PCIM and consider the asymptotic set called {\em the attractor} of $f$ and which is defined by the following equality:
\begin{equation}\label{LAMBDAN}
\Lambda:=\bigcap_{n\geq 1}\Lambda_n\qquad\text{where}\qquad\Lambda_1:=\overline{f(X\setminus \Delta)}
\quad\text{and}\quad \Lambda_{n+1}:=\overline{f(\Lambda_{n}\setminus \Delta)}\quad\forall\,n\geq 1.
\end{equation}
Note that this set does not depend on the particular definition of the map
at its  discontinuity points.
Also, as $\Lambda_n$ is compact, nonempty and $\Lambda_{n+1}\subset\Lambda_n$ for all $n\geq 1$, the attractor $\Lambda$ is compact and nonempty. Besides, as shown in  \cite{CGMU14}, the attractor contains the $\omega$-limit set of any point
of the set
\[
\X:=\bigcap_{n\geq0}f^{-n}(X\setminus\Delta).
\]


A general  result, which holds in any compact metric phase space,  is that the attractor of a piecewise contracting map consists of a finite number of periodic orbits, whenever it does not intersect the boundary of a contraction piece (see \cite{CGMU14}). For PCIM defined on a half-closed interval, Nogueira, Pires and Rosales proved moreover  that this periodic asymptotic behavior is generic in a metric sense and with a number  of periodic orbits which is bounded above by the number of contraction pieces \cite{NP15,NPR14,NPR16}. This generalizes and refines a previous result  by Br\'emont obtained in \cite{Br06}.

Periodic orbits are not the only possible asymptotic sets of PCIM.  In \cite{GT88}, Gambaudo and Tresser early studied the attractors of PCIM with $N=2$ contraction pieces.  Associating a rotation number to the map,  they  proved that the attractor is either a periodic orbit (rational rotation number) or a Cantor set
(irrational rotation number), and that the latter case corresponds to a quasi-periodic asymptotic dynamics with Sturmian complexity. It is in particular the case for the half-closed unit interval map $x\mapsto \lambda x+\mu \mod1$, for which the properties of the rotation number   as a function of $\lambda$
and $\mu\in [0,1)$ have been  studied in detail \cite{B93,BC99,C,LN18}.  For injective PCIM with $N\geq 2$ contractions pieces, it has been proved that  the complexity of the itinerary of any orbit is an eventually affine function \cite{CGM17, P18}. The growth rate of the complexity  is at most equal to $N-1$ and there are some examples of PCIM with such a maximal complexity \cite{CGM17}. In these particular examples, the attractor is a minimal Cantor set containing all the boundaries of the contraction pieces. But,  there is no general description of the topological structure and dynamical properties   of the attractor of PCIM with  arbitrary complexity and number of contraction pieces. The aim of this paper is to give such a description.

Before stating the hypothesis and our results, we fix the notations and give some definitions.
In the following, ${\mathcal O}(x) := \big\{f^n(x)\big\}_{ n \geq 0}$ denotes the forward orbit of a point $x\in X$ and it is said to be periodic if there exists $p\geq 1$ such that $f^p(x)=x$. The $\omega$-limit set
of a point $x\in X$ is denoted $\omega(x)$. We recall that $y\in\omega(x)$ if and only if there exits a subsequence
of ${\mathcal O}(x)$ which converges to $y$. In practice, we will only study the orbits and the $\omega$-limit sets of the points in $\X$\footnote{It is easy to see that the orbit of a point in $X\setminus\X$ eventually falls either in $\X$ or at a point of $\Delta$ which is periodic.} (nevertheless, the asymptotic sets may contain points of $\Delta$). This allows to disregard how the map is defined on $\Delta$, the relevant values being actually those of the set $D$.

%
%

\begin{definition}[Pseudo-invariant set] \em
 We say that $A\subset X$ is {\em pseudo-invariant} if for any $x\in A$ we have $\lim\limits_{y\to x^-}f(y)\in A$  or $\lim\limits_{y\to x^+}f(y)\in A$ .
\end{definition}

\noindent For PCIM the $\omega$-limit set of any point is nonempty and compact, but it is not necessarily invariant if it contains a discontinuity point.
However, we will see later that the attractor of a PCIM, as well as the $\omega$-limit set of any point of $\X$ are pseudo-invariant
sets. Note that if $A \subset X$ is pseudo-invariant, then $f(x)\in A$ for any $x\in A\setminus\Delta$ and  $A \cap \widetilde X$ is invariant.

\begin{definition} \em
We say that $A\subset X$ is {\em $\X$-minimal} if $\overline{\mathcal O(x)}=A$ for any $x\in A\cap\X$.
\end{definition}

\noindent  In some occasion, when a ``property" holds  for the intersection of a set $A\subset X$ with $\X$, we will say that the set $A$ is $\X$-``property". For instance, a set $A\subset X$ is $\X$-invariant if $f(A \cap \widetilde X)\subset A \cap \widetilde X$. Also, if $A$ and $B\subset X$ satisfy $A\cap B\cap\X=\emptyset$ we say that $A$ and $B$ are $\X$-disjoint.

Now, we state  Theorem \ref{PRINCIPAL}, which is  the main result  of this paper:

\begin{theorem}\label{PRINCIPAL}
Let $f:X\to X$ be a PCIM which is injective on each of its contraction pieces and such that $D\subset\X$. Then,
there exist two natural numbers $N_1$ and $N_2$ such that

\noindent 1)  The attractor $\Lambda$ of $f$ can be decomposed as follows:
\begin{equation}
\Lambda=\left(\bigcup\limits_{i=1}^{N_1}\mathcal O_i\right)\cup\left(\bigcup\limits_{j=1}^{N_2}K_j\right)\!,\label{DECOMP}
\end{equation}
where $\mathcal O_1,\mathcal O_2,\ldots,\mathcal O_{N_1}\subset\widetilde{X}$ are
periodic orbits and $K_1,K_2,\ldots,K_{N_2}$ are $\widetilde X$-minimal pseudo-invariant Cantor sets of $X$.

\noindent 2) For any $x\in\widetilde X$, either there exists $i\in\{1,\dots,N_1\}$ such that $\omega(x)=\mathcal O_i$ or there exists $j\in\{1,\dots,N_2\}$ such that $\omega(x)=K_j$.


\noindent 3) If $N_2\geq 1$, then for any $j\in\{1,\dots,N_2\}$ there exists $k\in\{1,\dots,N-1\}$ such that
\begin{equation}\label{GENERADOR}
c_k\in K_j\quad \text{and}\quad K_j=\overline{\mathcal{O}(d_k^+)}=\overline{\mathcal{O}(d_k^-)}.
\end{equation}

\noindent 4) If $N_2\geq 1$, then for any $j\in\{1,\dots,N_2\}$ and $k\in\{1,\dots,N-1\}$ such that $c_k\in K_j$ we have
\begin{equation}\label{GAP}
K_j=\overline{\mathcal{O}(d_k^+)}\quad\text{or}\quad K_j=\overline{\mathcal{O}(d_k^-)}.
\end{equation}
Moreover, if $c_k\in K_j$ does not belong to the boundary of a gap of $K_j$, then $\overline{\mathcal{O}(d_k^+)}=\overline{\mathcal{O}(d_k^-)}$.

\noindent 5) Finally, we have $1\leq N_1+N_2\leq \# D$ and $N_1+2N_2\leq 2(N-1)$. Moreover, if $f$ is increasing on each of its contraction pieces, then $N_1$ and $N_2$ also satisfy
$ N_1+N_2\leq N$.
\end{theorem}

Note that two different Cantor sets $K_i$ and $K_j$ of the decomposition \eqref{DECOMP} are necessarily $\X$-disjoint. Indeed, if there exits $y\in K_i\cap K_j\cap\X$, then $K_i=\overline{\mathcal{O}(y)}=K_j$, since $K_i$ and $K_j$ are $\X$-minimal. Therefore, Theorem \ref{PRINCIPAL} ensures a decomposition of the attractor $\Lambda$ into a finite number of topologically transitive, pseudo-invariant and $\X$-disjoint components. So we may call \eqref{DECOMP} the ``spectral decomposition'' of $\Lambda$ and each of its component a ``basic piece". Theorem \ref{PRINCIPAL} states also a dichotomy: a basic piece is either a periodic orbit in $\X$ or a $\X$-minimal Cantor set. This dichotomy does not hold when the phase space is not a subset of $\R$. Indeed, there are examples of PCM of compact subsets of $\R^n$ ($n\geq 2$) for which the attractor is a transitive countable infinite set, or an  interval, see \cite{CGMU14}.

Part 3) states that each Cantor piece must contain a border of a contraction piece. Part 4) states that a Cantor piece is given by the closure of the orbit of a (or both) one-sided limit(s) of the map at any point of $\Delta$ contained in the Cantor piece.
An estimation of the number of basic pieces is given by part 5). In particular,  we deduce that  $N_2\leq N-1$ and if $N_2=N-1$ then  $N_1=0$. If $N=2$, then $1\leq N_1+2N_2 \leq 2$, that is, the attractor  consists either of a single $\X$-minimal Cantor set, or of one or two periodic orbits. For any of these cases there exist  examples of PCIM with such an attractor \cite{B93,BC99,C,GT88, LN18}. So, the inequality is optimal at least for PCIM with two contraction pieces. If the map is increasing in each contraction pieces, then the number of basic pieces must satisfy the additional inequality $1\leq N_1+N_2\leq N$. In particular, it complements Theorem 1.1 of \cite{NPR16}, for $\lambda$-piecewise affine contractions which verify $\lambda\in(0,1)$ and $D\subset\widetilde{X}$. Finally, It is worth to mention that for  globally injective maps  we always have $N_1\leq N$, see \cite{NP15}.


In \cite{CGM17}, it is shown that for injective PCIM the complexity of the itinerary of any point in $\X$ is an eventually constant or affine function. As a consequence of Theorem \ref{PRINCIPAL}, we obtain that if $D\subset\X$ then the $\omega$-limit sets of the points with affine complexity  are $\X$-minimal Cantor sets.

\begin{remark} \em Note that the hypothesis of Theorem \ref{PRINCIPAL} only requires the PCIM being
injective in each contraction piece. Therefore, the theorem can be applied to
non-injective PCIM such as those of Figure \ref{FIG} $a)$. On the other hand, the collection of the contraction pieces of a
PCIM is not unique. The most natural and smallest one is the collection of the continuity pieces (for which $\Delta$ is the set of the discontinuity points of the map). However, Theorem \ref{PRINCIPAL} applies with any collection of contraction pieces, provided the pieces are chosen in such a way the map is injective in each of them.
For instance, if a PCIM has a finite number of local extrema, the hypothesis of the theorem are satisfied if we chose the contraction pieces of the map such that the set $\Delta$ contains all the points where the map has a local extremum (in addition to the discontinuity points), as in  Figure \ref{FIG} $b)$.
\em
\end{remark}

\begin{figure}[h!]
\begin{center}
\begin{minipage}[b]{6.5cm}
\newlength{\graphscale}
\setlength{\graphscale}{0.35cm}
\begin{tikzpicture} [domain=-10:10, scale=.7, >=stealth,x=\graphscale,y=\graphscale]
\node[] at (0,-11.5) {$(a)$};
\draw [very thick,black!50!white] (-10,0) -- (10,0);
\draw [very thick,black!50!white] (0,-10) -- (0,10);
\draw [very thick,black!50!white] (-10,-10) -- (-10,10);
\draw [very thick,black!50!white] (10,-10) -- (10,10);
\draw [very thick,black!50!white] (-10,-10) -- (10,-10);
\draw [very thick,black!50!white] (-10,10) -- (10,10);

\draw [very thick,dotted,gray!50!white] (-10,10) -- (10,-10);
\draw [very thick,dotted,gray!50!white] (-10,-10) -- (10,10);
\draw [ultra thick] (-1,3) parabola (2,4);
\draw [ultra thick] (5,4) -- (10,3);
\draw [ultra thick] (-10,-9) .. controls (-6,-8.5) and (-3,-8) .. (-1,-6);
\draw [ultra thick] (2,-5) -- (5,-7);

\draw [very thick,dotted,gray!50!white] (-1,7) -- (-1,-6);

\draw [very thick,dotted,gray!50!white] (2,4) -- (2,-5);
\draw [very thick,dotted,gray!50!white] (5,4) -- (5,-7);

\draw [very thick,dotted,gray!90!white] (-1,3) -- (-10,3);

\draw [very thick,dotted,gray!90!white] (2,-5) -- (10,-5);
\draw [very thick,dotted,gray!90!white] (5,-7) -- (10,-7);
\draw [very thick,dotted,gray!90!white] (5,4) -- (10,4);
\draw [very thick,dotted,gray!90!white] (2,4) -- (0,4);
\draw [very thick,dotted,gray!90!white] (-1,-6) -- (0,-6);

\opendot{-1}{3}
\opendot{5}{4}
\opendot{-1}{-6}
\opendot{2}{4}
\opendot{2}{-5}
\closeddot{5}{-7}
\closeddot{-1}{7}
\closeddot{2}{-4}

\draw [thick, decorate,decoration={brace,amplitude=6pt},xshift=0pt] (-1,-.01) -- (-10,-.01) node [black,midway,xshift=0pt,yshift=-12pt] {\tiny $X_1$};

\draw [thick, decorate,decoration={brace,amplitude=5pt},xshift=0pt] (2,-.01) -- (-1,-.01) node [black,midway,xshift=1pt,yshift=-12pt] {\tiny $X_2$};

\draw [thick, decorate,decoration={brace,amplitude=5pt},xshift=0pt] (5,-.01) -- (2,-.01) node [black,midway,xshift=0pt,yshift=-12pt] {\tiny $X_3$};

\draw [thick, decorate,decoration={brace,amplitude=5pt},xshift=0pt] (10,-.01) -- (5,-.01) node [black,midway,xshift=0pt,yshift=-12pt] {\tiny $X_4$};

\draw [thick, decorate,decoration={brace,amplitude=6pt},xshift=0pt] (-10.4,-10.01) -- (-10.4,10.01) node [black,midway,xshift=-12pt,yshift=0pt] {\footnotesize $X$};

\draw [thick, decorate,decoration={brace,amplitude=6pt},xshift=0pt] (-10,10.41) -- (10,10.41) node [black,midway,xshift=0pt,yshift=12pt] {\footnotesize $X$};

\node[] at (-9.34,.5) {\tiny $c_0$};
\node[] at (-1,.5) {\tiny $c_1$};
\node[] at (2,.5) {\tiny $c_2$};
\node[] at (5,.5) {\tiny $c_3$};
\node[] at (9.4,.5) {\tiny $c_4$};

\node[] at (-9.2,3.55) {\tiny $d_1^+$};
\node[] at (10.8,-4.8) {\tiny $d_2^+$};
\node[] at (10.8,-6.9) {\tiny $d_3^-$};
\node[] at (10.8,4.1) {\tiny $d_3^+$};
\node[] at (-.5,4.1) {\tiny $d_2^-$};
\node[] at (.78,-5.95) {\tiny $d_1^-$};
\end{tikzpicture}
\end{minipage}
\begin{minipage}[b]{6cm}
\setlength{\graphscale}{.35cm}
\begin{tikzpicture} [domain=-10:10, scale=.7, >=stealth,x=\graphscale,y=\graphscale]
\node[] at (0,-11.5) {$(b)$};
\draw [very thick,black!50!white] (-10,0) -- (10,0);
\draw [very thick,black!50!white] (0,-10) -- (0,10);
\draw [very thick,black!50!white] (-10,-10) -- (-10,10);
\draw [very thick,black!50!white] (10,-10) -- (10,10);
\draw [very thick,black!50!white] (-10,-10) -- (10,-10);
\draw [very thick,black!50!white] (-10,10) -- (10,10);
\draw [very thick,dotted,gray!50!white] (-10,10) -- (10,-10);
\draw [very thick,dotted,gray!50!white] (-10,-10) -- (10,10);
\draw [ultra thick] (-10,-6) .. controls (-7,-8) and (-5,-4) .. (-1.5,-7);
\draw [ultra thick] (-1.5,6.2) parabola (2,7) -- (5,6);
\draw [ultra thick] (5,-4) -- (10,-7);
\draw [very thick,dotted,gray!50!white] (-8.4,-9) -- (-8.4,-0);
\draw [very thick,dotted,gray!50!white] (-4.05,-5.86) -- (-4.05,0);
\draw [very thick,dotted,gray!50!white] (-1.5,-7) -- (-1.5,6.2);

\draw [very thick,dotted,gray!50!white] (5,6) -- (5,-4);
\draw [very thick,dotted,gray!50!white] (2,0) -- (2,7);
\draw [very thick,dotted,gray!50!white] (-10,-7) -- (-10,0);
\draw [very thick,dotted,gray!50!white] (10,-7) -- (10,0);

\draw [very thick,dotted,gray!90!white] (-8.4,-6.6) -- (-10,-6.6);
\draw [very thick,dotted,gray!90!white] (-4.04,-5.86) -- (0,-5.86);
\draw [very thick,dotted,gray!90!white] (2,7) -- (10,7);
\draw [very thick,dotted,gray!90!white] (5,6) -- (10,6);
\draw [very thick,dotted,gray!90!white] (5,-4) -- (10,-4);
\draw [very thick,dotted,gray!90!white] (-1.5,6.2) -- (-10,6.2);
\draw [very thick,dotted,gray!90!white] (-1.5,-7) -- (0,-7);

\opendot{-1.5}{6.2}
\opendot{5}{6}
\opendot{-8.4}{-6.6} 
\opendot{-1.5}{-7}

\closeddot{-1.5}{3}
\closeddot{-8.4}{-9}
\closeddot{-4.05}{-5.86} 
\closeddot{5}{-4}
\closeddot{2}{7}

\draw [thick, decorate,decoration={brace,amplitude=4pt},xshift=0pt] (-8.4,-.01) -- (-10,-.01) node [black,midway,xshift=0pt,yshift=-12pt] {\tiny $X_1$};
\draw [thick, decorate,decoration={brace,amplitude=5pt},xshift=0pt] (-4.05,-.01) -- (-8.4,-.01) node [black,midway,xshift=0pt,yshift=-12pt] {\tiny $X_2$};
\draw [thick, decorate,decoration={brace,amplitude=5pt},xshift=0pt] (-1.5,-.01) -- (-4.05,-.01) node [black,midway,xshift=0pt,yshift=-12pt] {\tiny $X_3$};

\draw [thick, decorate,decoration={brace,amplitude=5pt},xshift=0pt] (2,-.01) -- (-1.5,-.01) node [black,midway,xshift=2pt,yshift=-12pt] {\tiny $X_4$};

\draw [thick, decorate,decoration={brace,amplitude=5pt},xshift=0pt] (5,-.01) -- (2,-.01) node [black,midway,xshift=0pt,yshift=-12pt] {\tiny $X_5$};

\draw [thick, decorate,decoration={brace,amplitude=5pt},xshift=0pt] (10,-.01) -- (5,-.01) node [black,midway,xshift=0pt,yshift=-12pt] {\tiny $X_{6}$};
\draw [thick, decorate,decoration={brace,amplitude=6pt},xshift=0pt] (10.4,10.01) -- (10.4,-10.01) node [black,midway,xshift=12pt,yshift=0pt] {\footnotesize $X$};
\draw [thick, decorate,decoration={brace,amplitude=6pt},xshift=0pt] (-10,10.41) -- (10,10.41) node [black,midway,xshift=0pt,yshift=12pt] {\footnotesize $X$};

\node[] at (-10.6,.5) {\tiny $c_0$};
\node[] at (-8.4,.5) {\tiny $c_1$};
\node[] at (-4.05,.5) {\tiny $c_2$};
\node[] at (-1.5,.5) {\tiny $c_3$};
\node[] at (2,.5) {\tiny $c_4$};
\node[] at (5,.5) {\tiny $c_5$};
\node[] at (9.4,.5) {\tiny $c_6$};

\node[] at (8.0,7.6) {\tiny $d_4^-=d_4^+$};
\node[] at (-10.6,6.4) {\tiny $d_3^+$};
\node[] at (9.3,-3.4) {\tiny $d_5^+$};

\node[] at (9.3,5.5) {\tiny $d_5^-$};
\node[] at (.8,-6.9) {\tiny $d_3^-$};
\node[] at (2.2,-5.7) {\tiny $d_2^-=d_2^+$};
\node[] at (-9.4,-7.3) {\tiny $d_1^-=d_1^+$};

\end{tikzpicture}
\end{minipage}
\end{center}
\label{FIG}
\vspace*{-4ex}
\caption{Two examples of PCIMs.}
\end{figure}
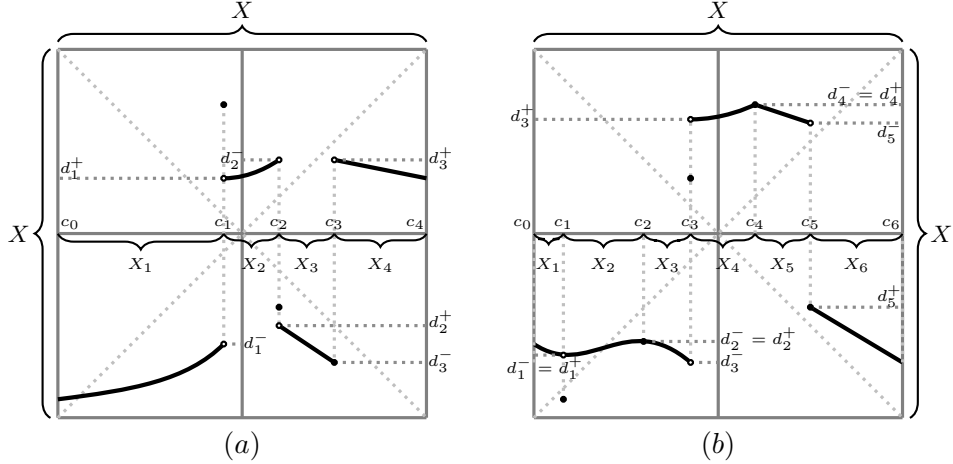

%

The paper is organized as follows. In Section \ref{ROUTE}, we give the route of the proof of Theorem \ref{PRINCIPAL}. That is, we prove Theorem \ref{PRINCIPAL}, but assuming Theorem \ref{Theorem3.5} which is stated without proof.
Then, to complete the proof of Theorem \ref{PRINCIPAL}, we give the proof of Theorem \ref{Theorem3.5} in Section \ref{PROOF}.

%

\section{Route of the proof of Theorem \ref{PRINCIPAL}}\label{ROUTE}


This section contains three theorems  (Theorem \ref{THDECOMP}, \ref{Theorem3.3}  and \ref{Theorem3.5}) which allow us to prove Theorem \ref{PRINCIPAL}. We will not always assume the hypothesis of Theorem \ref{PRINCIPAL} which states that $f$ is injective on each of its contraction pieces. We will explicitly mention this hypothesis in the statement of the results whose proof uses it. To prove Theorems \ref{THDECOMP} and \ref{Theorem3.3}, we will write the attractor $\Lambda$ as the intersection of collections of ``atoms", which are defined as follows:


\begin{definition}[Atoms]\em
Let $\mathcal P(X)$ be the power set of $X$ and for every $i\in\{1,\ldots,N\}$ consider
the map $F_i:\mathcal P(X)\to\mathcal P(X)$ defined by
\[
F_i(A)=\overline{f(A\cap X_i)} \qquad \forall \; A\in\mathcal P(X).
\]
Let $n\geq1$ and $(i_1,i_2,\ldots,i_n)\in\{1,\ldots,N\}^{n}$. We call the set
\[
A_{i_1, \ldots, i_{n-1}, i_n}:=F_{i_n}\circ F_{i_{n-1}}\circ\ldots\circ F_{i_1}(X)\,,
\]
an {\em atom of generation $n$} if it is nonempty. We denote by $\mathcal A_n$ the family of all the atoms of generation $n$.
\end{definition}

The atoms allow to study the attractor because the sets $\Lambda_n$ that define $\Lambda$ through \eqref{LAMBDAN} can also be written as
\[
\Lambda_n=\bigcup_{A\in\mathcal A_n}A\quad\forall\,n\geq 1.
\]

Also, if  $x\in\widetilde X$ and $\theta\in\{1,\ldots,N\}^{\mathbb N}$ is the itinerary of $x$, i.e. is the sequence such that  $f^n(x) \in X_{\theta_n}$ for all $n \in\N$, then $f^{t+n}(x)\in A_{\theta_{t},\theta_{t+1},\ldots,\theta_{t+n-1}}$ for every $t\geq0$ and $n\geq1$ (see \cite{CGM17}).

The basic properties of the atoms are the following ones: Any atom of generation $n$ is contained in an atom of generation $n-1$, precisely $A_{i_1, i_2,\ldots, i_n}\subset A_{i_2, i_3,\ldots, i_n}\subset \ldots \subset A_{i_n}$. Moreover,  if $f$ is piecewise contracting  with contracting constant $\lambda$,
then

\[
\max_{A\in\mathcal A_{n+1}}\diam(A)\leq\lambda\max_{A\in\mathcal A_{n}}\diam(A)\qquad\forall\,n\geq1,
\]
where $\diam(A)$ denotes the diameter of $A$. It  implies that the diameter of any atom of  gene\-ration $n$
is smaller than $\lambda^n\diam(X)$. Finally, in the case of PCIM, any atom is a compact interval.

\subsection{Decomposition and pseudo invariance of the attractor}

\begin{lemma}\label{OMEGA.INV}
If $x\in\widetilde X$ then $\omega(x)$ is nonempty, compact and pseudo-invariant.
\end{lemma}
\vspace*{-3ex}

\begin{proof}
From the compactness of the space $X$, and from definition of the $\omega$-limit set, $\omega(x)$ is nonempty and closed, hence compact. To prove that $\omega(x)$ is pseudo-invariant, we show that for any point $x_0 \in \omega(x)$  there exists $i\in\{1,\dots,N\}$ such that $f_i(x_0) \in \omega(x)$. Let $x_0 \in \omega(x)$ and
$\{t_j\}_{j\in \N}$ be a strictly increasing sequence such that $\lim\limits_{j\to\infty}f^{t_j}(x)=x_0$.
Then, there exist $i\in\{1,\dots,N\}$ such that $x_0\in \overline{X}_i$ and a subsequence $\{t_{j_k}\}_{k\in\N}$ of $\{t_j\}_{j\in\N}$ such that $f^{t_{j_k}}(x)\in X_i$ for all $k\in\N$. It follows that $f^{t_{j_k}+1}(x)=f_i(f^{t_{j_k}}(x))$ for any $k\in\N$ and by continuity of $f_i$ on $\overline{X}_i$ we have $\lim\limits_{k\to\infty}f^{t_{j_k}+1}(x)=f_i(x_0)\in\omega(x)$.
\end{proof}

\begin{lemma}\label{CUENCAPER}
If $f$ has a periodic point $x_0\in\X$, then there exists $\rho>0$ such that for any $x$ in the ball
 $B(x_0,\rho)$  of center $x_0$ and radius $\rho$ we have $\omega(x)={\mathcal O}(x_0)$. 
\end{lemma}
\begin{proof}
Let $\boldsymbol \nu$ denotes the distance between two subsets of $X$ and let $\rho:=\boldsymbol \nu(\mathcal O(x_0),\Delta)$. As the periodic point $x_0$ belongs to $\widetilde X$, we have $\rho>0$. Therefore, for every $n\in\N$ the ball $B(f^n(x_0),\rho)$ does not contain any point of $\Delta$, and for each $n\in\N$ it intersects only one of the contraction pieces. It follows that for any point $x\in B(x_0,\rho)$ we have
\[
|f^n(x_0)-f^n(x)|<\lambda^n\rho\qquad\forall n\in\N,
\]
where $\lambda\in(0,1)$ is the contracting rate of $f$. This implies that
\[
\boldsymbol\nu(\mathcal{O}(x_0),f^{n}(x))<\lambda^n\rho\qquad\forall n\in\N.
\]
Therefore,  if for some increasing sequence $\{s_n\}_{n\in\N}$ of natural number  $\{f^{s_n}(x)\}_{n\in\N}$ converges, then its limit is in $\mathcal{O}(x_0)$. In other words, $\omega(x)\subset\mathcal{O}(x_0)$. On the other hand, by invariance of $\omega(x)\cap\X$ we obtain that ${\mathcal O}(x_0)\subset\omega(x)$.
\end{proof}

%

The following Theorem \ref{THDECOMP} is the first key-point in the proof of Theorem \ref{PRINCIPAL}. It states that the attractor of a PCIM  is completely determined  by the $\omega$-limit sets of its one-sided limits at the points of  $\Delta$.

\begin{theorem}\label{THDECOMP} Suppose that $f$ is  injective on each of its contraction pieces  and  that $D\subset\X$.
Then,

\noindent 1) The attractor of $f$ can be written as
\begin{equation}\label{DECOMP2}
\Lambda=\bigcup_{d\in D}\omega(d).
\end{equation}

\noindent 2) For any periodic point $x_0\in \X$, there exists $d\in D^-\cup D^+$ such that $\mathcal{O}(x_0)=\omega(d)$, with $D^-:=\{d_1^-,\ldots,d_{N-1}^-\}$ and $D^+:=\{d_1^+,\ldots,d_{N-1}^+\}$. Moreover, if $f$ is  increasing  on each of its contraction pieces, then there exists $d^-\in D^-\cup \{d_N\}$ and $d^+ \in D^+\cup \{d_0\}$ such that $\mathcal{O}(x_0)=\omega(d^-)=\omega(d^+)$.

\end{theorem}

\begin{proof} Since the $\omega$-limit set of any point of $\widetilde{X}$ is contained in $\Lambda$, we
have that $\omega(d)\subset\Lambda$ for all $d\in D$. So, we have to prove that for any point $x_0\in\Lambda$ there exists $d\in D$ such that $x_0\in\omega(d)$ and that, besides,  $d$ can be chosen in $D^-\cup D^+$ if $x_0$ is periodic.

Define
\[
\mathcal U:=\bigcup_{d\in D}\mathcal O(d) \qquad\text{and}\qquad \mathcal U^*:=\bigcup_{d\in D^-\cup D^+}\mathcal O(d).
\]
Since $f$ is injective and continuous on each of its contraction pieces, for each $i\in\{1,\dots,N\}$ the continuous extension $f_i$ is strictly increasing or strictly decreasing. This implies that each atom of the first generation is a compact interval the end points of which are different and belong to the set
$D$. Moreover, at least one end point of each atom of the first generation belongs to $D^-\cup D^+$. Now, by induction on $n$, we prove that for every $n\geq2$ and every $A\in\mathcal A_n$ there exists $a, b  \in\mathcal U$  such that $A=[a,b]$, with $a\neq b$ and $a$ or $b$ in $\mathcal{U}^*$.
Assume that it is true for some $n\geq 1$ and let $A:=[a,b]\in\mathcal{A}_{n+1}$. Then, by definition of the atoms, there exists $A':=[a',b']\in\mathcal{A}_{n}$ and $i\in\{1,\dots,N\}$ such that $A=\overline{f(A'\cap X_i)}=f_i(\overline{A'\cap X_i})$. If $A'\subset X_i$, then $\{a,b\}=\{f(a'),f(b')\}$. If not, then $\overline{A'\cap X_i}$ is $[c_{i-1},b']$ or $[a',c_{i}]$ or $[c_{i-1},c_{i}]$ and
$\{a,b\}$ is $\{d_{i-1}^+,f(b')\}$ or $\{f(a'),d_{i}^-\}$ or $\{d_{i-1}^+,d_{i}^-\}$. In any case, $a\neq b$ belong to $\mathcal{U}$ and $a$ or $b \in\mathcal U^*$, because $f_i$ is injective and by the induction hypothesis.

Note that if $f$ is  increasing  on each of its contraction pieces, then we obtain with a similar  induction that
for every $n\geq 1$ and every $A\in\mathcal A_n$ there exist
\[
a\in\mathcal U^+:=\bigcup_{d\in D^+\cup\{d_0\}}\mathcal O(d)\qquad\text{and}\qquad
b\in\mathcal U^-:=\bigcup_{d\in D^-\cup\{d_N\}}\mathcal O(d)
\]
such that $A=[a,b]$, with $a\neq b$ and $a$ or $b$ in $\mathcal{U}^*$.

Now, let $x_0\in\Lambda$ and $\{A_n\}_{n\geq1}$ be a decreasing sequence of atoms such that $A_n\in\mathcal A_n$ for all $n\geq1$ and
\[
\{x_0\}=\bigcap_{n\geq1}A_n.
\]
The existence of $\{A_n\}_{ n \geq 1}$ is an immediate consequence of the properties of the atoms.

 Let $\{a_n\}_{n\geq1}$, \  $\{b_n\}_{n\geq1} \subset {\mathcal U}$ be such that $A_n=[a_n,b_n], \ \ a_n < b_n$ for each $n\geq1$. Since the diameter of $A_n$ tends to zero as $n$ goes to infinity, we deduce that $\lim\limits_{n\to\infty}a_n=\lim\limits_{n\to\infty}b_n=x_0$. Besides, as $a_n \neq b_n$ for all $n\geq 1$, one of the sequence $\{a_n\}_{n\geq1}$ or $\{b_n\}_{n\geq1}$, let us say $\{a_n\}_{n\geq1}$, is not eventually equal to $x_0$.

1) As $\{a_n\}_{n\geq1}$ converges to $x_0$ and is not eventually equal to $x_0$, it contains a subsequence $\{a_{n_k}\}_{k \geq 1}$ whose terms are all pairwise different. Since $\{a_n\}_{n\geq1}\subset\mathcal U$ and $\mathcal U$ is a finite union of orbits, we can choose $\{n_k\}_{k \geq 1}$ in such a way that for some $d \in D$ the subsequence $\{a_{n_k}\}_{k \geq 1}$ satisfies $a_{n_k}\in\mathcal O(d)$ for all $k\geq1$. Therefore, there exists a sequence $\{t_k\}_{k\geq1}$ such that
\[
a_{n_k}=f^{t_k}(d)\, \qquad\forall\,k\geq1.
\]
Since $a_{n_i}\neq a_{n_{j}}$ if $i\neq j$, there exists an increasing subsequence $\{t_{k_j}\}_{j\geq1}$ of $\{t_k\}_{k\geq1}$ such that
\[
\lim_{j\to\infty} f^{t_{k_j}}(d)=\lim_{k\to\infty} a_{n_k}=x_0,
\]
and we obtain that $x_0\in\omega(d)$. This proves that $\Lambda=\bigcup_{d\in D}\omega(d)$.

\noindent 2) Now suppose that $x_0\in\X$ is periodic and let $\rho:=\boldsymbol \nu(\mathcal O(x_0),\Delta)$, as in Lemma \ref{CUENCAPER}. Let $n_0\geq 1$ be such that the diameter of $A_{n_0}=[a_{n_0},b_{n_0}]$ is smaller than $\rho$. Then, applying Lemma \ref{CUENCAPER}, we obtain that $\mathcal{O}(x_0)=\omega(a_{n_0})=\omega(b_{n_0})$. Since $a_{n_0}$ or $b_{n_0}$ belongs to $\mathcal{U}^*$ we deduce that there exists  $d\in D^-\cup D^+$ such that $\omega(d)=\mathcal{O}(x_0)$. Now, if $f$ is  increasing  on each of its contraction pieces, then $a_{n_0}\in\mathcal{U}^+$ and $b_{n_0}\in\mathcal{U}^-$ and we can conclude that there exists $d^-\in D^-\cup \{d_N\}$ and $d^+ \in D^+\cup \{d_0\}$ such that $\mathcal{O}(x_0)=\omega(d^-)=\omega(d^+)$.
\end{proof}

Note that  Lemma \ref{OMEGA.INV} and Theorem \ref{THDECOMP} immediately imply that $\Lambda$ is a pseudo-invariant set. Later, we will use the following Lemma \ref{CAP.TILDE} which ensures that, besides, the $\omega$-limit set of any point of $\X$ and the attractor contain points of $\X$.

\begin{lemma}\label{CAP.TILDE}
If $D\subset\X$ and $\emptyset \neq G \subset X$ is pseudo-invariant, then $G\cap\widetilde X\neq\emptyset$.
\end{lemma}
\begin{proof}
Let $y\in G\setminus\widetilde X$. Let $t\geq0$ the first time such that $c_j:=f^t(y)\in G\cap\Delta$, for some $j\in\{1,\ldots,N-1\}$. Since $G$ is a pseudo-invariant set we have that   $d_j^+\in G$ or $d_j^-\in G$. Therefore, $G\cap\widetilde X\neq\emptyset$, because by hypothesis $d_j^-,d_j^+\in D \subset \widetilde X$.
\end{proof}


\subsection{Periodic and Cantor limit sets}

Here, we relate the asymptotic properties of any orbit in $\X$ to its recurrence properties in a neighborhood  of $\Delta$. Precisely, for each point $x \in \widetilde X$ we define the (maybe empty) set $\Delta_{lr}(x)\subset\Delta$ consisting of the points in $\Delta$ on which the orbit of $x$ accumulates  from both sides (see Definition \ref{definitionLeftRightVisits}). Then, we obtain the following dichotomic result:  if $\Delta_{lr}(x) = \emptyset$, then the $\omega$-limit set of $x$ is a periodic orbit in $\X$ (Theorem \ref{Theorem3.3}), and if  $\Delta_{lr}(x) \neq \emptyset$, then the  $\omega$-limit set of $x$ is a $\X$-minimal Cantor set (Theorem \ref{Theorem3.5}).

\begin{definition}[Left-right recurrently visited point]\label{definitionLeftRightVisits}
\em Let $i\in\{1,\dots,N-1\}$ and $x\in \widetilde X$. We say that $c_i\in\Delta$ is {\em left-right recurrently visited} (in short $lr$-recurrently visited) by the orbit of $x$, if there exists two strictly increasing sequences $\{l_j\}_{j\in\N}$ and $\{r_j\}_{j\in\N}$ of natural numbers such that
\[
f^{l_j}(x)\in X_{i}\;\mbox{ and }\;f^{r_j}(x)\in X_{i+1} \quad \forall \, j\in\N,
\quad
\text{and}\quad
c_i = \lim_{j\to\infty} f^{l_j}(x) = \lim_{j\to\infty} f^{r_j}(x).
\]
We denote by $\Delta_{lr}(x) \subset \Delta $ the set of points in $\Delta$ that are $lr$-recurrently visited by the orbit of $x$, and we denote by $\Delta_{lr}$ the set of points in $\Delta$ which are $lr$-recurrently visited by the orbit of some point in $\widetilde X$.
\end{definition}

\begin{remark} \em
Even if not immediate, it is not difficult to check that the Definition \ref{definitionLeftRightVisits} of the set $\Delta_{lr}(x)$ \em is equivalent \em to the combinatorial definition of the set of left-right recurrently visited discontinuities in \cite[Definition 2.8]{CGM17}.
\end{remark}

The basic properties of the left-right recurrently visited points are given in the following lemma:

\begin{lemma}\label{OMEGA1} Let $i\in\{1,\dots,N-1\}$, $x\in \widetilde X$ and suppose that $c_i\in\Delta_{\text{lr}}(x)$. Then, $c_i$, $d_i^+$ and $d_i^-$ belong to $\omega(x)$. If moreover $D\subset\X$, then $\overline{\mathcal O(d_i^-)}\cup\overline{\mathcal O(d_i^+)}\subset\omega(x).$

\end{lemma}
\vspace*{-3ex}
\begin{proof} By definition of $\omega$-limit set and of left-right recurrently visited point, if $c_i\in\Delta_{lr}(x)$ then $c_i\in\omega(x)$. We can show that this implies that $d_i^+$ and $d_i^-$ belong to $\omega(x)$ with a similar proof as that of Lemma \ref{OMEGA.INV}. If we suppose moreover that $D\subset\X$, then $\mathcal O(d_i^-)$ and $\mathcal O(d_i^+)\subset\omega(x)$, since $\omega(x)\cap\X$ is invariant by pseudo-invariance of $\omega(x)$. The desired inclusion follows from the compactness of $\omega(x)$.
\end{proof}



\begin{theorem}[Periodic $\omega$-limits]\label{Theorem3.3} Suppose that $f$ is such that $D\subset\X$. Let $x\in\widetilde X$, then  $\omega(x)$ is a periodic orbit contained in $\widetilde X$  if and only if $\Delta_{lr}(x)=\emptyset$.
\end{theorem}
\begin{proof}
Let $x\in\widetilde X$. Suppose that $\omega(x)$ is contained in $\X$. Then, it follows from Lemma \ref{OMEGA1} that $\Delta_{lr}(x)=\emptyset$. Indeed,  if $\Delta_{lr}(x)\neq\emptyset$ then there is some point of $\Delta$ in $\omega(x)$ and therefore $\omega(x)$ is not contained in $\X$. Now we suppose that $\Delta_{lr}(x)=\emptyset$ and we prove that $\omega(x)$ is a periodic orbit contained in $\widetilde X$.

We first show that under the hypothesis $\Delta_{lr}(x)=\emptyset$, the itinerary of $x$ is eventually periodic.  Let $\eta\in\{1,\ldots,N\}^{\mathbb N}$ be the itinerary of $x$ and for any $n\geq 1$, define the set
\[
L_n(\eta):=\big\{(\eta_{t},\eta_{t+1},\ldots,\eta_{t+n-1})\in\{1,\ldots,N\}^{n}\;{:}\;t\geq0\big\}
\]
of the words of size $n$ contained in $\eta$. The function $p_{\eta}$ defined for any $n\geq1$ by $p_\eta(n):= \# L_n(\eta)$ is the
complexity function of $\eta$. By the Morse-Hedlund's Theorem \cite{MH40}, if $p_\eta$ is eventually constant, then $\eta$ is eventually periodic.
Obviously $\#L_n(\eta)\leq\#L_{n+1}(\eta)$. So, we have to show that if
$\Delta_{lr}(x)=\emptyset$, then there exists $n_0\geq 1$ such that the converse inequality also holds, and therefore
\begin{equation}\label{ECONSTANT}
\#L_n(\eta)=\#L_{n+1}(\eta)\quad\forall\, n\geq n_0.
\end{equation}
To that aim, recall that $f^{t+n}(x)\in A_{\eta_{t},\eta_{t+1},\ldots,\eta_{t+n-1}}$ for every $t\geq0$ and $n\geq1$.

First, let us prove that  for any $n\geq 1$ we have
\begin{equation}\label{INCLN}
L_{n+1}(\eta)\subset\bigcup_{(i_1,\ldots,i_n)\in L_{n}(\eta)}\big\{(i_1,\ldots,i_n,i_{n+1})\ {:}\ \exists\,t\geq0 \ : \ f^{t+n}(x)\in A_{i_1,\ldots,i_n}\cap X_{i_{n+1}}\big\}.
\end{equation}
Indeed, if $(i_1,\ldots,i_{n+1})\in L_{n+1}(\eta)$, then there exists $t\geq0$ such that
\[
(\eta_t,\ldots,\eta_{t+n})=(i_1,\ldots,i_{n+1})
\]
and by definition of $L_n(\eta)$ and of the itinerary $\eta$ we have that $(i_1,\ldots,i_n)\in L_n(\eta)$ and $f^{t+n}(x)\in X_{i_{n+1}}$. As $f^{t+n}(x)\in A_{i_1,\ldots,i_n}$, we conclude that there exists $t\geq0$ such that
\[
(i_1,\ldots,i_n)\in L_n(\eta)\quad\text{ and }\quad f^{t+n}(x)\in A_{i_1,\ldots,i_n}\cap X_{i_{n+1}},
\]
that is, $(i_1,\ldots,i_{n+1})$ belongs to the set of the right hand side of the inclusion \eqref{INCLN}.

Now, if $\Delta_{lr}(x)=\emptyset$, then there exists $\epsilon>0$ such that
\[
\mathcal O(x)\cap (c_i-\epsilon, c_i)=\emptyset\quad \text{or}\quad\mathcal O(x)\cap (c_i,c_i+\epsilon)=\emptyset\quad\forall\, i\in\{1,\ldots,N-1\}.
\]
Also, we know that there exists $n_0\geq 1$ such that  $\diam A<\epsilon$ for all $A\in \mathcal A_{n}$ and $n\geq n_0$. Therefore, if $n\geq n_0$, then for any  $(i_1,\ldots,i_n)\in L_n(\eta)$ fixed, we have that
\[
\#\big\{(i_1,\ldots,i_n,i_{n+1})\ {:}\ \exists\,t\geq0 \ : \ f^{t+n}(x)\in A_{i_1,\ldots,i_n}\cap X_{i_{n+1}}\big\}=1.
\]
Thus, from \eqref{INCLN}, we conclude that $\#L_{n+1}(\eta)\leq\#L_n(\eta)$  for all $n\geq n_0$, which ends the proof of  \eqref{ECONSTANT}.


Since we have proved that the itinerary $\eta$ of $x$ is eventually periodic, we know that there exist $t\geq0$ and $p\geq 1$ such that $\theta:=\sigma^{t}(\eta)$ is a periodic sequence with period $p$, where $\sigma$ denotes the shift map in the space of sequences with an   alphabet of $N$ symbols. Let $y:=f^{t}(x)$. As $\omega(x)=\omega(y)$, to finish the proof, we show that $\omega(y)$ is a periodic orbit contained in $\X$.

Since $\theta$ is the itinerary of $y$, we deduce that
\[
f^{k+p}(y)\in A_{\theta_k,\ldots,\theta_{k+p-1}} \quad \forall \, k\in\{0,1,\ldots,p-1\}.
\]
More generally,
\begin{equation}\label{CODY}
f^{k+jp}(y)\in A_{\theta_k,\ldots,\theta_{k+j p-1}} \qquad \forall\,j\geq1,\;\;\forall\,k  \in\{0,1,\ldots,p-1\}.
\end{equation}
Besides,
\[
A_{\theta_k,\ldots,\theta_{k+p-1}}\supset A_{\theta_k,\ldots,\theta_{k+2 p-1 }}\supset \ldots \supset A_{\theta_k,\ldots,\theta_{k+j p-1 }}\supset \ldots
\]
is a decreasing sequence of (nonempty compact) atoms whose diameters converge to zero. Then, there exists $x_k^*\in X$ such that
\begin{equation}
   \label{eqn1}
   \bigcap_{j\geq1}A_{\theta_k,\ldots,\theta_{k+j p-1 }}=\{x_k^*\}.
\end{equation}
Taking all the values of $k \in \{0,1, \ldots, p-1\}$, we conclude that
\begin{equation}
    \label{eqn0a}
   \{x_0^*,x_1^*,\ldots,x_{p-1}^*\}\subset\omega(y).
\end{equation}

Now, let us prove the converse inclusion. If $z\in\omega(y)$ then  there exists a strictly increasing sequence $\{m_n\}_{n\in\N}$ such that $f^{m_n}(y)$ converges to $z$ when $n$ goes to infinity. Let $\{q_n\}_{n\in\N}\in\N^{\N}$ and $\{r_n\}_{n\in\N}\in\{0,1, \dots, p-1\}^{\N}$ be such that
\[
m_n=q_np+r_n \quad\forall\, n\in\N.
\]
Since $\{m_n\}_{n\in\N}$ is strictly increasing and $\{r_n\}_{n\in\N}$ takes only a finite number of values, the sequence of integer quotients $\{q_n\}_{n\in\N}$ is also strictly increasing. Besides, there exist $\{n_j\}_{j\in\N}$ and $k\in \{0,1,\ldots,p-1\}$ such that $r_{n_j}=k$ for all $j\in\N$. We deduce that
\[
z=\lim_{n\to\infty}f^{m_n}(y) = \lim_{n\to\infty}f^{q_np+r_n}(y) = \lim_{j\to\infty}f^{q_{n_{\!j}}p+k}(y) \; \in \;  \bigcap_{j\geq1} A_{\theta_k,\ldots,\theta_{k+q_{n_{\!j}}p - 1}} = \{x_k^*\}.
\]
Therefore, we have proved that $z\in\{x_0^*,x_1^*,\dots,x_{p-1}^*\}$ for any $z\in\omega(y)$. Together with  (\ref{eqn0a}), this implies that
\begin{equation}
    \label{eqn0c}\omega(y) = \{x_0^*,x_1^*,\ldots,x_{p-1}^*\}.
\end{equation}

Finally, let us prove that $\omega(y)$ is a periodic orbit contained in $\widetilde X$. By  Lemma \ref{CAP.TILDE}, we know that $\omega(y)\cap\widetilde X\neq\emptyset$. Thus,  there exists $k\in\{0,1,\dots,p-1\}$ such that $x_k^*\in\widetilde X$. This implies that the distance $\rho$ between
$x_k^*$ and any element of $\Delta$ is positive. Since the diameter of the atoms decreases with their generation, there exists $j_0$ such that
\[
\diam(A_{\theta_k,\ldots,\theta_{k+j p-1 }})<\rho\quad\forall\, j\geq j_0.
\]
From equality (\ref{eqn1}) we deduce that for any $j\geq j_0$ the atom
$A_{\theta_k,\ldots,\theta_{k+j p-1 }}$ is contained in the same contraction piece than $x_k^*$. On the other hand,
by \eqref{CODY} and the definition of itinerary
\[
f^{k+jp}(y)\in A_{\theta_k,\dots,\theta_{k+jp-1}}\cap X_{\theta_{k+jp}} \qquad \forall\,j\geq1.
\]
This implies that for any $j\geq j_0$ the atom $A_{\theta_k,\dots,\theta_{k+jp-1}}$  is contained in
$X_{\theta_{k+jp}}$. Therefore,
\[
f(A_{\theta_k,\dots,\theta_{k+j p-1 }})=\overline{f(A_{\theta_k,\dots,\theta_{k+jp-1}}\cap X_{\theta_{k+jp}})}= A_{\theta_k,\dots,\theta_{k+jp}}\subset A_{\theta_{k+1},\dots,\theta_{k+jp}}\qquad\forall\, j\geq j_0.
\]
Now we can conclude from equality (\ref{eqn1}) that
\[
\big\{f(x_k^*)\big\} \subset \bigcap_{j\geq j_0} f(A_{\theta_k,\ldots,\theta_{k+j p-1 }})\subset \bigcap_{j\geq j_0} A_{\theta_{k+1},\ldots,\theta_{k+jp}}=\big\{x_{k+1 \,(\!\!\bmod p)}^*\big\}.
\]
Then, $f(x_k^*)=x_{k+1 \,(\!\!\bmod p)}^*\in\X$, since $x^*_k\in\widetilde X$. So we can repeat the same argument for all the iterates of $x_k^*$ to obtain $f^l(x_k^*)=x_{k+l\,(\!\!\bmod p)}^*\in\widetilde X$ for all $l\geq1$. We conclude that $\omega(y)=\{x_0^*,x_1^*,\ldots,x_{p-1}^*\}=\omega(x) $   is a periodic orbit contained in $\widetilde X$, as wanted.
\end{proof}

Now, we state the complementary results of Theorem \ref{Theorem3.3}. Its proof needs a larger development which is done in Section \ref{PROOF}.



\begin{theorem}[Cantor $\omega$-limits]\label{Theorem3.5} Suppose that $f$ is injective on each of its contraction pieces  and  that $D\subset\X$.  Then,  for any $x\in\widetilde X$, $\Delta_{lr}(x)\neq\emptyset$ if and only if  $\omega(x)$ is a $\X$-minimal Cantor set.
\end{theorem}

\begin{proof} See Section \ref{PROOF}.
\end{proof}



\subsection{Proof of Theorem \ref{PRINCIPAL}} \label{subsectionReduction}

Now we prove Theorem \ref{PRINCIPAL} assuming Theorem \ref{Theorem3.5}.


\noindent 1) For any $d\in D $, either $\Delta_{lr}(d)=\emptyset$ and  applying Theorems \ref{Theorem3.3}
it follows that $\omega(d)$ is a periodic orbits contained in $\X$, or $\Delta_{lr}(d)\neq\emptyset$ and applying Theorem \ref{Theorem3.5} we deduce that $\omega(d)$ is a $\X$-minimal Cantor set. So, we can rewrite \eqref{DECOMP2} as follows:
\begin{eqnarray}
\Lambda=\bigcup_{d\in D}\omega(d)=\left(\bigcup\limits_{i=1}^{N_1}\mathcal O_i\right)\cup\left(\bigcup\limits_{j=1}^{N_2}K_j\right)\!,\label{DECOMP3}
\end{eqnarray}
where $\mathcal O_1,\mathcal O_2,\ldots,\mathcal O_{N_1}\subset\widetilde{X}$ are periodic orbits and $K_1,K_2,\ldots,K_{N_2}$ are $\widetilde{X}$-minimal Cantor sets.
As $D\subset\X$,   Lemma \ref{OMEGA.INV} ensures that the Cantor sets are pseudo-invariant.

\vspace{1ex}

\noindent 2) Now, let us prove that the $\omega$-limit set of any point $x\in\widetilde X$ coincides either with one periodic orbit ${\mathcal O}_i$, or with one Cantor set $K_j$. First, recall that the $\omega$-limit set $\omega(x)$ of any point $x\in \widetilde X$ satisfies $\omega(x)\cap\widetilde X\neq\emptyset$ (see Lemma \ref{CAP.TILDE}). Then, there exists $y \in \omega(x)\cap \widetilde X$. Since $\omega(x) \subset \Lambda$, from Theorem \ref{THDECOMP} we deduce that there exists $d\in D$ such that $y\in\omega(d)$, so $y \in \omega(x) \cap \omega(d) \cap \widetilde X$. Besides, $x,d\in \widetilde X$, so we can apply Theorems \ref{Theorem3.3} and \ref{Theorem3.5} to deduce that both $\omega(x)$ and $\omega(d)$ are  $\widetilde{X}$-minimal sets. Therefore,
\[
\omega(x)=\overline{\mathcal{O}(y)}=\omega(d).
\]
This proves that $\omega(x)$ coincides with some set of the decomposition \eqref{DECOMP3}, and it also proves that the sets of the decomposition \eqref{DECOMP3} are all pairwise $\widetilde{X}$-disjoint.
We conclude that,  for any $x\in\widetilde X$, either there exists $i\in\{1,\dots,N_1\}$ such that $\omega(x)=\mathcal O_i$, or there exists $j\in\{1,\dots,N_2\}$ such that $\omega(x)=K_j$.

\vspace{1ex}
\noindent 3) Suppose that $N_2\geq 1$. Let $j\in\{1,\dots,N_2\}$ and let $d\in D$ be such that $\omega(d)=K_j$.
Since $\omega(d)=K_j$, according to Theorem \ref{Theorem3.5} there exists $k\in\{1,\dots,N-1\}$ such that $c_k\in\Delta_{lr}(d)$. From Lemma \ref{OMEGA1}, it follows that $c_k$, $d_k^-$ and $d_k^+\in\omega(d)=K_j$.
As $D\subset\X$ and $K_j$ is $\widetilde{X}$-minimal, we have that
$\overline{\mathcal{O}(d_k^-)}=K_j=\overline{\mathcal{O}(d_k^+)}$.

\vspace{1ex}
\noindent 4) Let $j\in\{1,\dots,N_2\}$ and $k\in\{1,\dots,N-1\}$ be such that $c_k\in K_j$. Since $K_j$ is pseudo-invariant we deduce  that $d_k^-$ or $d_k^+\in K_j$. As $D\subset\X$ and $K_j$ is $\widetilde{X}$-minimal, we have that  $K_j=\overline{\mathcal{O}(d_k^+)}$ or $K_j=\overline{\mathcal{O}(d_k^-)}$. Suppose moreover that
$c_k\in K_j$ does not belong to the boundary of a gap of $K_j$. If $K_j=\overline{\mathcal{O}(d_k^+)}$, then $c_k\in\Delta_{lr}(d_k^+)$ and  from Lemma \ref{OMEGA1} it follows that $d_k^-\in K_j$. Since $K_j$ is $\widetilde{X}$-minimal, we obtain that $\overline{\mathcal{O}(d_k^-)}=K_j$. An analog proof allows us to show that
$K_j=\overline{\mathcal{O}(d_k^+)}$ in the case where $K_j=\overline{\mathcal{O}(d_k^-)}$.

\vspace{1ex}
\noindent 5) From \eqref{DECOMP3} it follows that immediately that  $1\leq N_1+N_2\leq \#D$. Now, we show that
\[
N_1+2N_2\leq 2(N-1).
\]
Let  $d_1',d_2',\dots,d'_{2(N-1)}$ be such that
\[
d'_{2k-1}:=d_k^-\quad\text{and}\quad d'_{2k}:=d_k^+\qquad\forall\, k\in\{1,\dots,N-1\}.
\]
Consider the sets
\[
C_1:=\left\{l\in\{1,\dots,2(N-1)\}\ : \ \Delta_{lr}(d'_l)=\emptyset\right\} \quad\text{and}\quad C_2:=\left\{l\in\{1,\dots,2(N-1)\}\ : \ \Delta_{lr}(d'_l)\neq\emptyset\right\}.
\]
 Let $\mathcal O_1,\mathcal O_2,\ldots,\mathcal O_{N_1}\subset\widetilde{X}$ and $K_1,K_2,\ldots,K_{N_2}$
be the periodic orbits and the $\widetilde{X}$-minimal Cantor sets of the decomposition \eqref{DECOMP3}, respectively.

From part 2) of Theorem \ref{THDECOMP}, we know that for every $i\in\{1,\dots,N_1\}$ there exists $\l(i)\in C_1$ such that
\[
\mathcal O_i=\omega(d'_{\l(i)}).
\]
The function $l:\{1,\dots,N_1\}\to C_1$ defined by $i\mapsto l(i)$ being injective we have that  $N_1\leq\#C_1$.

From part 3), we know that for every $j\in\{1,\dots,N_2\}$ there exists an odd number $\ell(j) \in C_2$ such that
\[
K_j=\overline{\mathcal{O}(d_{\ell(j)}')}=\overline{\mathcal{O}(d_{\ell(j)+1}')}.
\]
The function $(j,s)\mapsto \ell(j)+s$ from the set $\{1,\dots,N_2\}\times\{0,1\}$ to the set $C_2$ being injective, we obtain  that $2N_2\leq\#C_2$, which together with $N_1\leq\#C_1$ gives
\[
N_1+2N_2\leq \#C_1+\#C_2=\#(C_1\cup C_2)=2(N-1).
\]

Finally, suppose  $f$ is  increasing  on each of its contraction pieces. Let $d_0':=d_0$, $d'_{2N-1}:=d_N$ and
\[
C:=\left\{l\in\{0,1,\dots,2N-1\}\ : \ \Delta_{lr}(d'_l)=\emptyset\right\}.
\]
Then, from part 2) of Theorem \ref{THDECOMP}, we know that for every $i\in\{1,\dots,N_1\}$ there exists an odd number $\l_1(i)\in C$ and an even number $l_2(i)\in C$ such that
\[
\mathcal O_i=\omega(d'_{l_1(i)})=\omega(d'_{l_2(i)}).
\]
The function $(i,s)\mapsto l_s(i)$ from the set $\{1,\dots,N_1\}\times\{1,2\}$ to the set $C$ being injective, we obtain  that $2N_1\leq\#C$, which together with $2N_2\leq\#C_2$ gives
\[
2N_1+2N_2\leq \#C+\#C_2=\#(C\cup C_2)=2N.
\]
This ends the proof of Theorem \ref{PRINCIPAL} assuming Theorem \ref{Theorem3.5}.

\section{Proof of Theorem \ref{Theorem3.5}}\label{PROOF}

All along this section we assume that $f$ is such that $D\subset\X$ and $\Delta_{lr}\neq\emptyset$. In other words, we suppose that $f$ has  at least one point $c\in\Delta$ which is left-right recurrently visited by the orbit of some point  $x\in\X$. We already know by Theorem \ref{Theorem3.3} that this implies that the $\omega$-limit set of such point $x$ is not a periodic orbit in $\X$. In Subsection \ref{PREM}, we will first show a stronger preliminary result: this $\omega$-limit set cannot contain a periodic point belonging to $\X$. It will imply that the orbits of the one sided limits of $f$ at the points of $\Delta_{lr}(x)$ do not accumulate neither at periodic points contained in $\X$. These preliminary results will be used in Subsection \ref{KMINIMAL} to prove that the $\omega$-limit set of some particular points of $D$ is $\X$-minimal.

In Subsection \ref{CLASSES}, we construct a partial order in a quotient set of $\Delta_{lr}$. This allows us to define minimal classes of points of $\Delta$, which are the minimal nodes in the Hasse graph of such a partial order (Definition \ref{definitionMinimalClass}). The study of the asymptotic dynamics of a point $x$ satisfying  $\Delta_{lr}(x)\neq\emptyset$ can be done by analyzing the minimal classes. Indeed, in Subsection \ref{ATRACT}, we show that if $\Delta_{lr}(x)\neq\emptyset$ then $\omega(x)$ is equal to $\omega(d)$ where $d\in D$ is a one sided limit of $f$ at a point of
$\Delta_{lr}(x)$ belonging to a minimal class (Theorem \ref{IMPORTANT}). In Subsection \ref{KMINIMAL}, we study the $\omega$-limit sets of the elements of $D$ associated to  a minimal class and show that they are $\X$-minimal Cantor set (Theorem \ref{Theorem4.7}). These two results allow to complete the proof of Theorem \ref{Theorem3.5}.

\subsection{Preliminary results}\label{PREM}

\begin{lemma}\label{NOTPERIODIC}
Let $x\in\widetilde X$ and suppose that $f$ has a periodic point $p\in\X$. If $p\in\omega(x)$, then $\omega(x)={\mathcal O}(p)$.
\end{lemma}

\begin{proof} It is a direct consequence of Lemma \ref{CUENCAPER}.
\end{proof}

\begin{corollary}\label{NOT.EMPTY2}
Let $x\in \X$ and $i\in\{1,\dots,N-1\}$.  If $c_i\in\Delta_{lr}(x)$ then $\omega(x)\cap\X$, $\omega(d_i^{+})\cap\X$ and $\omega(d_i^{-})\cap\X$ do not contain any periodic point.  \end{corollary}
\begin{proof}
Suppose that $c_i\in\Delta_{lr}(x)$, then from Theorem \ref{Theorem3.3} we deduce that $\omega(x)$ is not a periodic orbit of $\X$. Therefore, by Lemma \ref{NOTPERIODIC} it does not contain periodic point in $\X$. On the other hand, since $D\subset\X$, by Lemma \ref{OMEGA1} we have that $\omega(d_i^+)\cup\omega(d_i^-)\subset\omega(x)$. It follows that neither $\omega(d_i^+)$ nor $\omega(d_i^-)$ contains a periodic point  in $\X$.
\end{proof}

\begin{corollary}\label{NOT.EMPTY}
 Let $i\in\{1,\dots,N-1\}$ and $c_i\in\Delta_{lr}$.  Then, $\Delta_{lr}(d_i^{-})\neq\emptyset$ and $\Delta_{lr}(d_i^{+})\neq\emptyset$.  \end{corollary}
\begin{proof}
Suppose that $c_i\in\Delta_{lr}$, then by Definition \ref{definitionLeftRightVisits}, there exists $x\in\widetilde X$ such that $c_i\in\Delta_{lr}(x)$. From Corollary \ref{NOT.EMPTY2} we deduce that $\omega(d_i^+)$ and $\omega(d_i^-)$ are not a periodic orbit of $\X$.  Applying Theorem \ref{Theorem3.3} we deduce that $\Delta_{lr}(d_i^{-})\neq\emptyset$ and $\Delta_{lr}(d_i^{+})\neq\emptyset$.
\end{proof}

\subsection{Equivalence classes in $\Delta_{lr}$ and their partial order}\label{CLASSES}

Here we introduce an equivalence relation in $\Delta_{lr}$ and a partial order in the resulting quotient space. This allows to identify some classes of points of $\Delta_{lr}$ which are minimal elements with respect to the partial order.  These minimal classes will be of special importance
to study the non-periodic asymptotic dynamics.

Before defining our equivalence relation, let us prove the following lemma:

\begin{lemma}\label{OMEGA}
Let $x\in\widetilde X$. If there exist $i$ and $k\in\{1,\dots,N-1\}$ such
that  $c_i\in\Delta_{lr}(d_k^+)$ and $c_k\in\Delta_{lr}(x)$, then $c_i\in\Delta_{lr}(x)$.
\end{lemma}
\begin{proof}
If $c_k\in\Delta_{lr}(x)$, then ${\mathcal O}(d^+_k)\subset\omega(x)$, see Lemma \ref{OMEGA1}. This implies that the orbit of $x$ accumulates at any point of the orbit of $d^+_k$. On the other hand, we have $c_i\in\Delta_{lr}(d^+_k)$. This means that the orbit of $d^+_k$ accumulates at $c_i$ from the left and from the right. Joining the two latter assertions, we conclude that the orbit of $x$ also accumulates at $c_i$ from the left and from the right. In other words, $c_i\in\Delta_{lr}(x)$.
\end{proof}

\begin{definition}\label{DefinitionSIMP}
\em Let $i$ and $j\in\{1,\dots,N-1\}$ be such that $c_i$ and $c_j\in\Delta_{lr}$. We write $c_i\simp c_j$ and we say that $c_i$ and $c_j$ are related if and only if
\[
 c_i=c_j \quad\text{or}\quad c_i\in\Delta_{lr}(d_j^+)\text{ and }c_j\in\Delta_{lr}(d_i^+).
\]
\end{definition}



%
%
%


\begin{lemma} The relation $\simp$ is an equivalence relation on $\Delta_{lr}$.
\end{lemma}

\begin{proof} The identity and symmetric properties follow immediately from the definition of the relation $\sim^+$. So, it is left to prove the transitive property. Let $i,j$ and $k\in\{1,\dots,N-1\}$ be such that $c_i, c_j$ and $c_k\in\Delta_{lr}$. Let us suppose that $c_i \simp c_j$ and $c_j \simp c_k$ and let us show that  $c_i \simp c_k$. This assertion holds trivially if $c_i = c_j$ or $c_j = c_k$. If  $c_i \neq c_j$ and $c_j \neq c_k$, by definition of the relation $\simp$, we have
\[
c_i \in \Delta_{lr}(d_j^+), \quad  c_j\in\Delta_{lr}(d_k^+), \quad c_k \in \Delta_{lr}(d_j^+)\quad\text{and}\quad
c_j \in \Delta_{lr}(d_i^+).
\]
Applying Lemma \ref{OMEGA}, we conclude that $c_i \in \Delta_{lr}(d_k^+)$  and $c_k \in \Delta_{lr}(d_i^+)$,
which  implies $c_i \simp c_k$.
\end{proof}

For any point $c \in \Delta_{lr}$, we let $[c]$ denote the equivalence class of $c$. In order to
define an order relation on the (non-empty) set $\Delta_{lr}/\!\!\simp$ of the equivalence classes of $\Delta_{lr}$, we first prove the following lemma.


\begin{lemma}\label{NODEPEND} Let $i$ and $j\in\{1,\dots,N-1\}$ be such that $c_i$ and $c_j\in\Delta_{lr}$. If  $c_{i} \in \Delta_{lr}(d_{j }^+)$,  then
$c_{i'}\in\Delta_{lr}(d_{j'}^+)$ for all $i'$ and $j'\in\{1,\dots,N-1\}$ such that $c_{i'}\in[c_i]$ and $c_{j'}\in[c_j]$.
\end{lemma}

\begin{proof}
Suppose that $c_{i'}\simp c_i$ and $c_{j'}\simp c_j$. First, assume that $c_{i'}\neq c_i$ and $c_{j'}\neq c_j$. In this case, the definition of $\simp$ implies that
\[
c_{i'} \in \Delta_{lr}(d_i^+) \quad \mbox{and} \quad  c_j\in\Delta_{lr}(d_{j'}^+).
\]
Applying Lemma \ref{OMEGA} for $c_{i'}\in\Delta_{lr}(d_i^+)$ and $c_{i}\in\Delta_{lr}(d_{j}^+)$, we obtain  that $c_{i'}\in \Delta_{lr}(d_{j}^+)$. Applying once again the same lemma but for $c_{i'}\in \Delta_{lr}(d_{j}^+)$ and $c_j\in\Delta_{lr}(d_{j'}^+)$ we conclude that $c_{i'}\in \Delta_{lr}(d_{j'}^+)$, as wanted. To obtain the same
result in the complementary case $c_{i'} = c_i$ or $c_{j'} = c_j$, we can use similar arguments.
\end{proof}

%
%
%
%

\begin{definition}\label{DefinitionORD}\em  Let $i$ and $j\in\{1,\dots,N-1\}$ be such that $c_i$ and $c_j\in\Delta_{lr}$. We  define the relation $\ord$ between the equivalence classes $[c_i]$ and $[c_j]$ in $\Delta_{lr}/\!\! \simp$ by
\[
[c_i]\ord [c_j]\qquad
\text{if and only if}\qquad [c_i]=[c_j] \quad\text{or}\quad  c_{i }\in\Delta_{lr}(d_{j }^+) .
\]
\end{definition}

\noindent Note that Lemma \ref{NODEPEND} proves that the above definition is well posed, since it is independent of the choice of the elements $c_i,c_j$ in the equivalence classes $[c_i]$ and $[c_j]$.

%


\begin{lemma}\label{PARTIAL}
$\big(\Delta_{lr}/\!\!\simp,\ord\big)$ is a partially ordered set.
\end{lemma}

\begin{proof}  Take $[c]$, $[c']$ and $[c'']\in\Delta_{lr}/\!\!\simp$. Let $i,j$ and $k\in\{1,\dots,N-1\}$ be such that $[c_i]=[c], [c_j]=[c']$ and $[c_k]=[c'']$.

\noindent {\em Reflexive property}: It follows trivially from Definition \ref{DefinitionORD}.

\noindent {\em Antisymmetric property}. Suppose $[c_i]\ord[c_j]$ and $[c_j]\ord[c_i]$. Then, from Definition \ref{DefinitionORD}, it follows that either $[c_i] = [c_j]$, and we are done, or $c_i \in \Delta_{lr}(d_j^+)$ and  $c_j \in \Delta_{lr}(d_i^+)$. In this last case, we deduce from Definition \ref{DefinitionSIMP}
that  $c_i\simp c_j$, which implies that $[c_i]=[c_j]$.

\noindent {\em Transitive property}: Suppose $[c_i]\ord [c_j]$ and $[c_j]\ord [c_k]$. If $[c_i]= [c_j]$ or $[c_j]= [c_k]$, then $[c_i]\ord [c_k]$.
Otherwise, we have $c_{i }\in\Delta_{lr}(d_{j}^+)$ and $c_{j }\in\Delta_{lr}(d_{k}^+)$. Applying the Lemma \ref{OMEGA}, we obtain
$c_{i}\in\Delta_{lr}(d_{k}^+)$ and we conclude that $[c_i]\ord[c_k]$.
\end{proof}

\begin{definition}[{{\bf Minimal classes}}] \em
\label{definitionMinimalClass} Let $[c] \in \Delta_{lr}/\!\! \simp$. We say that $[c]$ is a \em minimal class \em if it is a minimal element of the partially ordered set  $\big(\Delta_{lr}/\!\!\simp,\ord\big)$. In other words, $[c]$ is a minimal class if for every $[c']\in\Delta_{lr}/\!\! \simp$ such that $[c']\ord[c]$ we have  $[c'] = [c]$.
\end{definition}

\noindent It is well known that  any finite partially ordered set  has at least one  minimal element. Since  our partially ordered set $\big(\Delta_{lr}/\!\!\simp,\ord\big)$ is finite, it always has minimal classes.



\begin{Proposition}\label{DIAGRAM}
{\bf a)} Let $j\in\{1,\dots,N-1\}$ be such that  $c_j\in\Delta_{lr}$. Then, there exists $i\in\{1,\dots,N-1\}$ such that $[c_i]$ is a minimal class
and  $[c_i]\ord [c_j]$.
%

\noindent {\bf b)} Let  $[c]\in\Delta_{lr}/\!\!\simp$ and $i\in\{1,\dots,N-1\}$ be such that  $c_{i}\in[c]$.  Then, $[c]$ is a minimal class if and only if  $c_{i} \in \Delta_{lr}(d_j^+)$ for every $j\in\{1,\dots,N-1\}$ such that $c_j\in\Delta_{lr}(d_{i}^+)$.
\end{Proposition}


\begin{proof}
\noindent{\bf a)} For any Hasse graph of a partial order on a finite nonempty set, and for any of its nodes, say $j$, there exists at least one minimal node, say $i$, smaller or equal than $j$. Applying this assertion to the partially ordered set $\big(\Delta_{lr}/\!\!\simp,\ord\!\big)$, we deduce that for all $[c_j]\in\Delta_{lr}/\!\!\simp$, there exists at least one minimal class $[c_i]$ such that $[c_i]\ord[c_j]$.

\noindent {\bf b)} Let  $[c]\in\Delta_{lr}/\!\!\simp$ and let $i\in\{1,\dots,N-1\}$ be such that  $c_{i}\in[c]$.

Suppose that $[c]$ is a minimal class. If $c_j\in\Delta_{lr}(d_{i}^+)$ for some $j\in\{1,\dots,N-1\}$, then $[c_j]\ord[c_i]$. This implies that $[c_j]=[c_i]$, because $[c_i]=[c]$ and $[c]$ is a minimal class. It follows that $c_i\simp c_j$ and therefore we have that $c_{i} \in \Delta_{lr}(d_j^+)$.

Now suppose that $c_{i} \in \Delta_{lr}(d_j^+)$ for all $j\in\{1,\dots,N-1\}$ such that $c_j\in\Delta_{lr}(d_{i}^+)$. Let $j\in\{1,\dots,N-1\}$ be such
that $[c_j]\ord[c]$. Since $[c]=[c_i]$, to prove that $[c]$ is a minimal class, we have to show that $[c_j]=[c_i]$. By definition of $\ord$ either  $[c_j]=[c_i]$, and we are done, or $c_j\in\Delta_{lr}(d_i^+)$.
By hypothesis, the second case implies that $c_{i} \in \Delta_{lr}(d_j^+)$. It follows that $c_{i} \simp c_j$ and therefore $[c_j]=[c_i]$.
\end{proof}

\subsection{Asymptotic dynamics and minimal classes}\label{ATRACT}

In this section, we show that the non-periodic asymptotic dynamics is supported on
the closure of the orbits of the one-sided limits of the map at its minimal class points in $\Delta$.
Precisely, we will prove the following theorem:


\begin{theorem}\label{IMPORTANT}  If $x\in\widetilde X$ and $\Delta_{lr}(x)\neq\emptyset$, then there exists $i\in\{1,\dots,N-1\}$ such that $c_i\in\Delta_{lr}(x)$ and  $[c_i]$ is a minimal class. Moreover, if $f$ is injective on each of its contraction pieces, then  $\omega(x)=\omega(d_i^+)=\overline{\mathcal{O}(d_i^+)}$.
\end{theorem}


\noindent Note that we can define equivalence classes and a partial order $\ordm$ based on the left-sided limits
of the map $f$ at the points of $\Delta_{lr}$, just exchanging the superscript $+$ and $-$ in our definitions and proofs. Therefore,
the same Theorem \ref{IMPORTANT} is also true for the left-sided limits of the map. Actually, in the next
subsection, Theorem \ref{Theorem4.7}  will precise and (re)prove  this assertion.

To prove Theorem \ref{IMPORTANT}, we need the following two lemmas:

\begin{lemma}\label{NECESS1}
Let $x\in\widetilde X$. There exists $\epsilon(x)>0$ such that if for some $l, r\in\N$ and $c\in\Delta$ we have $f^l(x) \in (c-\epsilon(x),c)$  and $f^r(x)\in (c,c + \epsilon(x))$,
then $c\in \Delta_{lr}(x)$.
\end{lemma}


\begin{proof} If $\Delta_{lr}(x)=\Delta$, then the Lemma is true for any $\epsilon(x)>0$.  Now suppose that $\Delta\setminus  \Delta_{lr}(x)\neq\emptyset$. By Definition \ref{definitionLeftRightVisits}, we have that for any $c\in\Delta\setminus  \Delta_{lr}(x)$ there exists $\epsilon_c >0$  such that
$f^t(x)\notin(c-\epsilon_c ,c )$ for all $t\in\N$ or $f^t(x)\notin(c,c + \epsilon_c)$ for all $t\in\N$. Now, we define
$$
\epsilon(x):=\min_{c\in \Delta\setminus\Delta_{lr}(x)}\!\epsilon_c>0.
$$
Suppose that there exist $l, r\in\N$ and $c\in\Delta$ such that
\[
 f^l(x)\in (c-\epsilon(x),c) \quad\text{and}\quad f^r(x) \in (c,c+\epsilon(x)).
\]
Then, by definition of $\epsilon(x)$, we must have that $c\notin\Delta\setminus  \Delta_{lr}(x)$. Therefore,  $c\in  \Delta_{lr}(x)$.
\end{proof}

\begin{lemma}\label{NECESS2}
Suppose that $f$ is injective on each of its contraction pieces and let $x\in\widetilde X$ be such that $\Delta_{lr}(x)\neq\emptyset$. If there
exist  $i,j\in\{1,\dots,N-1\}$ such that
\begin{equation}
\label{eqn18b}
c_i \in \Delta_{lr}(d_j^+)\cap\Delta_{lr}(x)\quad\text{and}\quad c_j \in \Delta_{lr}(d_i^+)\cap\Delta_{lr}(x),
\end{equation}
then, there exist $\epsilon_0>0$, $m_0\geq0$,  and two sequences $\{\alpha_k\}_{k\in\N}$ and $\{\beta_k\}_{k\in\N}$ such that

\noindent 1)  $\{\alpha_k\}_{k\geq 1}$  is a subsequence of $ \mathcal O(d_i^+)$ and  $\{\beta_k\}_{k\geq 1}$  is a subsequence of $ \mathcal O(d_j^+)$,

\noindent 2)  the closed interval $I_k$ whose endpoints are $\alpha_k$ and $\beta_k$ satisfies
\begin{equation}
\label{eqn15}
|\beta_k-\alpha_k| < \lambda^k\epsilon_0  \quad\text{and}\quad  f^{m_0+k}(x)\in I_k   \qquad \forall\,k \in\N.
\end{equation}
\end{lemma}

\begin{proof}  First we construct $\epsilon_0$, $m_0$, $\alpha_0$ and $\beta_0$. Let  $\epsilon(d_i^+)$ and $\epsilon(d_j^+)$ be as in Lemma \ref{NECESS1} and
 \begin{equation}
\label{eqn12a}
 0< \epsilon_1 := \min\big\{|c-c'|\;{:}\; c, c' \in \Delta,\, c \neq c' \big\}.
\end{equation}
We define $\epsilon_0$ as $\epsilon_0:=\min\{\epsilon(d_i^+),\epsilon(d_j^+), \epsilon_1\}$.

As $c_i \in \Delta_{lr}(d_j^+)\cap \Delta_{lr}(x)$, from  Definition \ref{definitionLeftRightVisits}, we deduce that there exists $n_0 \geq 0$ and $m_0\geq 0$
such that
\[
f ^{m_0}(x) \in \big (c_i,f^{n_0}(d_j^+)\big)\subset (c_i,c_i+\epsilon_0)\subset X_{i+1}.
\]
Denote $\alpha_0:=c_i$ and $ \beta_0:=f^{n_0}(d_j^+)$. Since $d_j^+\in\X$ we have that $\alpha_0\neq\beta_0$ and the relation above implies that
\begin{equation}
\label{eqn19b}
0<|\beta_0-\alpha_0| < \epsilon_0  \quad\text{and}\quad  f^{m_0}(x)\in (\alpha_0,\beta_0)\subset X_{i+1},
\end{equation}
which shows that \eqref{eqn15} holds for $k=0$.


Now, we show by induction that for any $k\geq 1$  there exist two  points $\alpha_k$ and $\beta_k \in X$ that satisfy the following properties:
\begin{equation} \label{eqn20a}
\alpha_k \in {\mathcal O}(d_i^+),\quad \beta_k \in {\mathcal O}(d_j^+), \quad |\beta_k - \alpha_k| < \lambda^{k} \epsilon_0\quad\text{and}\quad f^{m_0 + k} (x) \in I_{k},
\end{equation}
where $I_k$ is the compact interval whose endpoints are $\alpha_k$ and $\beta_k$.

Let us show \eqref{eqn20a} for $k=1$. Let $I_0:=[\alpha_0,\beta_0]$. According to \eqref{eqn19b} we have  that  $I_0\subset \overline{X_{i+1}}$, and as $f_{i+1}$ is $\lambda$-Lipschitz, we deduce that  $I_1:=f_{i+1}(I_0)$ is a compact interval of size smaller than $\lambda\epsilon_0$ such that  $f^{m_0+1}(x)\in I_1$. As  $f_{i+1}$ is a strictly monotonic function, the endpoints of $I_1$ are
\begin{equation}\label{a1b1}
\alpha_1:=d_i^+\quad\text{and}\quad \beta_1:=f(\beta_0)
\end{equation}
and belong to ${\mathcal O}(d_i^+)$ and ${\mathcal O}(d_j^+)$, respectively. It follows that  \eqref{eqn20a} holds for $k=1$.

Assume that \eqref{eqn20a} holds  for some $k \geq 1$. We discuss two cases:

\noindent{\em Case 1: } There is no point of $\Delta$ in the interval $I_k$. Then, $f|_{I_k}$ is a $\lambda$-Lipschitz strictly monotonic function and using the induction hypothesis \eqref{eqn20a} we obtain that
\begin{equation}
\label{eqn21a}
\alpha_{k+1} := f(\alpha_k) \quad\text{and}\quad \beta_{k+1} := f(\beta_k)
\end{equation}
satisfy \eqref{eqn20a} replacing $k$ by $k+1$.

\noindent{\em Case 2: } There exists a point $c_{\ell} \in I_k\cap\Delta$. First, note that such a point $c_{\ell}$ is unique, because of (\ref{eqn12a}) and  $$\mbox{length}(I_k) = |\alpha_k - \beta_k| < \lambda^k\epsilon_0\leq \lambda^k\epsilon_1.$$
Second, note that $$c_{\ell} \in \mbox{int}(I_k), $$ because  the endpoints  $\alpha_k$ and $\beta_k$ of $I_k$  belong to $\widetilde X$.  Indeed, by induction hypotesis  $\alpha_k \in {\mathcal O}(d_i^+) \subset \widetilde X$ and $\beta_k \in {\mathcal O}(d_j^+) \subset \widetilde X$ (recall that $D\subset\X$). Therefore,
\[
\alpha_k, \beta_k \in (c_{\ell}- \lambda^k\epsilon_0, c_{\ell} + \lambda^k\epsilon_0)
\]
and one of the two points $\alpha_k$, $\beta_k$ is at left of $c_{\ell}$, and the other one is at right of $c_{\ell}$. Without loss of generality we will suppose that
\begin{equation}\label{SIDES}
\alpha_k \in(c_{\ell}- \lambda^k\epsilon_0, c_{\ell}) \quad\text{and}\quad \beta_k \in (c_{\ell},c_{\ell} + \lambda^k\epsilon_0).
\end{equation}

Now we show that $c_{\ell} \in \Delta_{lr}(\alpha_k)\cap\Delta_{lr}(\beta_k)$. Recall that by \eqref{eqn18b} we have $c_j\in\Delta_{lr}(d_i^+)$ and that by Lemma \ref{OMEGA1} this implies that
$\mathcal{O}(d_j^+)\subset\omega(d_i^+)$. As $\alpha_k \in {\mathcal O}(d_i^+)$ we have $\omega(\alpha_k)=\omega(d_i^+)$ and as  $ \beta_k \in {\mathcal O}(d_j^+)$
we deduce from the right hand relation  of \eqref{SIDES} that there exists $n>0$ such that
\[
f^n(\alpha_k) \in (c_{\ell},c_{\ell}+ \lambda^k\epsilon_0).
\]
Then, from  the left hand relation  of \eqref{SIDES}, the definition of $\epsilon_0$, and Lemma \ref{NECESS1}, it follows that $c_{\ell} \in \Delta_{lr}(\alpha_k)$. Analogously, using that $c_i\in\Delta_{lr}(d_j^+)$, we obtain  $c_{\ell} \in \Delta_{lr}(\beta_k)$. This ends the proof of $c_{\ell} \in \Delta_{lr}(\alpha_k)\cap\Delta_{lr}(\beta_k)$.

Now, let us construct $\alpha_{k+1}$ and $\beta_{k+1}$.  By \eqref{eqn20a} we have $f^{m_0+k}(x)\in[\alpha_k,\beta_k]$. Suppose that $f^{m_0+k}(x)\in(c_{\ell},\beta_k]$. Since $c_{\ell} \in \Delta_{lr}(\alpha_k)$, there exists $r>0$ such that
\[
f^r(\alpha_k)\in(c_{\ell},f^{m_0+k}(x)).
\]
Therefore the interval $[f^r(\alpha_k),\beta_k]$ satisfies the same properties \eqref{eqn20a} as the interval $I_k$ and moreover does not contain a point in $\Delta$.
So, we can use the same proof as in Case 1, to show that
\begin{equation}\label{akbk1}
\alpha_{k+1}:=f^{r+1}(\alpha_k)\quad\text{and}\quad\beta_{k+1}:=f(\beta_k)
\end{equation}
satisfy \eqref{eqn20a} replacing $k$ by $k+1$. Now, if we suppose that $f^{m_0+k}(x)\in[\alpha_k,c_{\ell})$, then using this time that $c_{\ell} \in \Delta_{lr}(\beta_k)$ we obtain that there exists $l>0$ such that
\[
f^l(\beta_k)\in(f^{m_0+k}(x),c_{\ell}).
\]
Therefore, for the same reason as for the case where $f^{m_0+k}(x)\in(c_{\ell},\beta_k]$ we conclude  that
\begin{equation}\label{akbk2}
\alpha_{k+1}:=f(\alpha_k)\quad\text{and}\quad\beta_{k+1}:=f^{l+1}(\beta_k)
\end{equation}
satisfy \eqref{eqn20a} replacing $k$ by $k+1$.

We have constructed  by induction  two sequences $\{\alpha_k\}_{k\geq 1}$ and $\{\beta_k\}_{k\geq 1}$  satisfying \eqref{eqn20a} for all $k\geq 1$, which are moreover
subsequences of ${\mathcal O}(d_i^+)$ and ${\mathcal O}(d_j^+)$, respectively (see, \eqref{a1b1},  \eqref{eqn21a},  \eqref{akbk1} and  \eqref{akbk2}).
\end{proof}

Note that in Lemma \ref{NECESS2}, as well as in its following corollary,  the integers $i$ and $j$ are not necessarily different. As a consequence, their results can be applied even if $\Delta_{lr}(x)$ contains only one point.

\begin{corollary}\label{CORNECESS2}
Suppose that  $f$ is injective on each of its contraction pieces  and let $x\in\widetilde X$ be such that $\Delta_{lr}(x)\neq\emptyset$. If $i,j\in\{1,\dots,N-1\}$ are such that
\[
c_i \in \Delta_{lr}(d_j^+)\cap\Delta_{lr}(x)\quad\text{and}\quad c_j \in \Delta_{lr}(d_i^+)\cap\Delta_{lr}(x),
\]
then, $\omega(x)=\omega(d_i^+)=\omega(d_j^+)$.
\end{corollary}

\begin{proof} Applying Lemma \ref{OMEGA1}, we obtain immediately that
$\omega (d_i^+)\subset\omega(x)$ and $\omega (d_j^+)\subset\omega(x)$. Now, according to Lemma \ref{NECESS2}, there  exist $m_0\geq 0, \epsilon_0>0$,
a subsequence $\{\alpha_k\}_{k\geq 1}$  of $ \mathcal O(d_i^+)$ and a subsequence $\{\beta_k\}_{k\geq 1}$ of $ \mathcal O(d_j^+)$ such that
\begin{equation}
\big|f^{m_0+k}(x)-\alpha_k\big|\leq \lambda^k\epsilon_0\quad\text{and}\quad\big|f^{m_0+k}(x)-\beta_k\big|\leq \lambda^k\epsilon_0\qquad\forall\,k\geq1.
\label{eqninter}\end{equation}
Let $y\in\omega(x)$ and $\{k_n\}_{n\in\N}$ be an increasing sequences such that $\lim\limits_{n\to\infty}f^{k_n}(f^{m_0}(x))=y$. Then, \eqref{eqninter} implies that $\lim\limits_{n\to\infty}\alpha_{k_n}=y=\lim\limits_{n\to\infty}\beta_{k_n}$ and therefore
$y\in\omega(d_i^+)\cap\omega(d_j^+)$. So, we have proved that $\omega(x) \subset \omega (d_i^+)$ and $\omega(x) \subset \omega (d_j^+)$.
\end{proof}


\begin{proof}[Proof of Theorem \ref{IMPORTANT}] Let $x\in\X$ and suppose that $\Delta_{lr}(x)\neq\emptyset$. Then, there exists $k\in\{1,\dots,N-1\}$
such that $c_k\in\Delta_{lr}(x)$. Applying part {\em a)} of Proposition \ref{DIAGRAM}, we know that there exists $i\in\{1,\dots,N-1\}$ such that $[c_i]\in\Delta_{lr}/\!\!\simp$ is a minimal class and $[c_i]\ord[c_k]$. From Definition \ref{DefinitionORD}, it follows that either $c_i\in\Delta_{lr}(d_k^+)$ and Lemma \ref{OMEGA} ensures that $c_i\in\Delta_{lr}(x)$, or $[c_i]=[c_k]$ and we conclude also that $c_i\in\Delta_{lr}(x)$. We have proved that there exists a point
\[
c_i\in\Delta_{lr}(x),
\]
whose equivalence class $[c_i]$ is minimal.

Applying Corollary \ref{NOT.EMPTY}, we deduce that there exists $j\in\{1,\dots,N-1\}$ such
that $c_j\in\Delta_{lr}(d_i^+)$. Using once more Lemma \ref{OMEGA}, we obtain that
\[
c_j\in\Delta_{lr}(d_i^+)\cap\Delta_{lr}(x).
\]
On the other hand, as the class of $c_i$  is a minimal class, $c_j\in\Delta_{lr}(d_i^+)$ also implies that $c_i\in\Delta_{lr}(d_j^+)$, see part b) of Proposition \ref{DIAGRAM}.  It follows that
\[
c_i\in\Delta_{lr}(d_j^+)\cap\Delta_{lr}(x).
\]
Therefore, the hypothesis of Corollary \ref{CORNECESS2} are verified and $\omega(x)=\omega(d_i^+)$. Besides, as $c_i\in\Delta_{lr}(x)$, by Lemma \ref{OMEGA1}, we have
\[
\overline{\mathcal{O}(d_i^+)}\subset \omega(x)=\omega(d_i^+)\subset \overline{\mathcal{O}(d_i^+)},
\]
which ends the proof of Theorem \ref{IMPORTANT}.
\end{proof}


\subsection{End of proof of Theorem \ref{Theorem3.5}}\label{KMINIMAL}

In this section, we study the orbits of the points of $D$ corresponding to the minimal classes of $\Delta_{lr}/\!\!\simp$.
By Theorem \ref{IMPORTANT}, we know that these orbits determine all the non-periodic asymptotic dynamics. Among other results, we show that the closure of such an orbit is a $\X$-minimal Cantor set, which together with Theorem \ref{IMPORTANT} will achieve the proof of Theorem \ref{Theorem3.5}.

\begin{lemma}\label{PROPSMIN} Let $i\in\{1,\dots,N-1\}$ and suppose that  $[c_i]\in\Delta_{lr}/\!\!\simp$ is a minimal class. Then, for any $x\in\omega(d^+_i)\cap\X$ we have  $c_i\in\Delta_{lr}(x)$ and
\[
\omega(x)=\overline{\mathcal O(x)}=\omega(d_i^+)=\overline{\mathcal O(d_i^+)}.
\]
\end{lemma}

\begin{proof} Let $x\in\omega(d_i^+)\cap\X$. Since $\omega(d_i^+)\cap\X$ is invariant, we have that
\begin{equation}\label{INC}
\omega(x)\subset\overline{\mathcal{O}(x)}\subset\omega(d_i^+).
\end{equation}
As $c_i\in\Delta_{lr}$, from Corollary \ref{NOT.EMPTY2} we know that $\omega(d_i^+)\cap\X$ does not contain any periodic point, and therefore, by \eqref{INC}, $\omega(x)\cap\X$ does not either. It follows by Theorem \ref{Theorem3.3} that there exists $j\in\{1,\dots, N-1\}$ such $c_j\in\Delta_{lr}(x)$.

Moreover, still by \eqref{INC}, we have that $\mathcal{O}(x)\subset\overline{\mathcal{O}(d_i^+)}$, which allows us to deduce  that $c_j\in\Delta_{lr}(d^+_i)$. Since $c_i$ is of minimal class,
we must have that $c_i\in\Delta_{lr}(d^+_j)$, which together with $c_j\in\Delta_{lr}(x)$ implies by  Lemma \ref{OMEGA} that $c_i\in\Delta_{lr}(x)$.

Once we know that $c_i\in\Delta_{lr}(x)$, we deduce from Lemma \ref{OMEGA1} that
$\overline{\mathcal{O}(d_i^+)}\subset\omega(x)$ and using \eqref{INC} we obtain that
\[
\overline{\mathcal{O}(d_i^+)}\subset\omega(x)\subset\overline{\mathcal{O}(x)}\subset\omega(d_i^+)\subset\overline{\mathcal{O}(d_i^+)}.
\]
\end{proof}

\begin{theorem}\label{Theorem4.7}
Let $i\in\{1,\dots,N-1\}$ and suppose that $[c_i]\in\Delta_{lr}/\!\!\simp$ is a minimal class. Then, $K_i:=\omega(d_i^+)$ is a $\widetilde X$-minimal  Cantor set. Moreover, if $f$ is injective on each of its contraction pieces, then
for any $k\in\{1,\dots,N-1\}$ such that $[c_{i}]\ord[c_k]$, we have
\begin{equation}\label{KO}
c_k\in K_i\qquad\text{and}\qquad K_i=\overline{\mathcal{O}(d_k^+)}=\overline{\mathcal{O}(d_k^-)}.
\end{equation}
\end{theorem}

\begin{proof} Let $i\in\{1,\dots,N-1\}$, $K_i:=\omega(d_i^+)$ and suppose that $[c_i]\in\Delta_{lr}/\!\!\simp$ is a minimal class.

\noindent {\em $K_i$ is $\X$-minimal:} It is a direct consequence of Lemma \ref{PROPSMIN}. It also proves
that $K_i$ is a compact set.



\noindent {\em $K_i$ is a perfect set:} Let $y\in K_i$. As $K_i$ is pseudo invariant (see Lemma \ref{OMEGA.INV}),
there exists $x\in K_i\cap\X$ (see Lemma \ref{CAP.TILDE}) and $\overline{\mathcal{O}(x)}=K_i$.
As $c_i\in\Delta_{lr}$ and $D\subset\X$, from Corollary \ref{NOT.EMPTY2} we deduce that $K_i\cap\X$ does not contain periodic points. Therefore $\mathcal{O}(x)\subset\X$ does not contain periodic points and there exists
$n_0\in\N$ such that $y\notin \mathcal{O}(f^{n_0}(x))$. As $\mathcal{O}(f^{n_0}(x))$ is dense in $K_i$, there exists $\{y_n\}_{n\in\N}\subset\mathcal{O}(f^{n_0}(x))\subset K_i\setminus\{y\}$ which converges to $y$.

\noindent{\em $K_i$ is totally disconnected:} In \cite[Theorem 5.2]{CGMU14} it is proved  that, if $f$ is a piecewise contracting map on a one dimensional compact space $X$,
then its attractor $\Lambda$ is totally disconnected. As any $\omega$-limit set is contained in $\Lambda$, we conclude that $K_i$ is also totally disconnected.

Now, let $k\in\{1,\dots,N-1\}$ be such  that $[c_i]\ord[c_k]$. As $c_k\in\Delta_{lr}$, there exists $x\in\X$ such that
\begin{equation}\label{CKX}
c_k\in\Delta_{lr}(x).
\end{equation}
According to Theorem \ref{IMPORTANT}, this implies that there exists $i'\in\{1,\dots,N-1\}$ such  that $[c_{i'}]$ is a minimal class and $\omega(x)=\omega(d_{i'}^+)$. We have proved previously  that if $[c_{i'}]$ is a minimal class, then
$K_{i'}:=\omega(d_{i'}^+)$ is a $\X$-minimal Cantor set. Therefore, Lemma \ref{OMEGA1} and \eqref{CKX}
imply that
\[
c_k,d_k^+,d_k^-\in K_{i'}\qquad\text{and}\qquad K_{i'}=\overline{\mathcal{O}(d_k^+)}=\overline{\mathcal{O}(d_k^-)}.
\]
To finish the proof of the theorem, we only have to show that $K_{i'}=K_i$. To this end note that
\begin{equation}\label{CKX2}
c_i\in\Delta_{lr}(x).
\end{equation}
Indeed, \eqref{CKX2} follows from $[c_i]\ord[c_k]$, \eqref{CKX} and Lemma \ref{OMEGA}. We deduce from \eqref{CKX2} and Lemma
\ref{OMEGA1} that $\omega(d_i^+)\subset\omega(x)$, that is
\[
K_i\subset K_{i'}.
\]
Since $K_i$ and $K_{i'}$ are both $\X$-minimal, and $K_i\cap\X\neq\emptyset$ we conclude that  $K_{i'}=K_i$.
\end{proof}

Now, we can prove Theorem \ref{Theorem3.5}, which, as said in Subsection \ref{subsectionReduction}, will also complete the proof of Theorem \ref{PRINCIPAL}.

\begin{proof}[Proof of Theorem \ref{Theorem3.5}.] Suppose that $f$ is injective on each of its contracting pieces  and that $D\subset\X$. Let $x\in\widetilde X$. If $\Delta_{lr}(x)\neq\emptyset$, then according to Theorem \ref{IMPORTANT},
there exists $i\in\{1,\dots,N-1\}$ such that $[c_i]$ is a minimal class and $\omega(x)=\omega(d^+_i)$. Using
Theorem \ref{Theorem4.7}, we deduce that $\omega(x)$ is a $\X$-minimal Cantor set. Reciprocally, if
$\omega(x)$ is a $\X$-minimal Cantor set, then $\omega(x)$ is not a periodic orbit and we obtain
from Theorem \ref{Theorem3.3} that $\Delta_{lr}(x)\neq\emptyset$.
\end{proof}

Note that Theorem \ref{Theorem4.7} allows the proof of Theorem \ref{Theorem3.5}, but also states in addition, through \eqref{KO}, that all the points in $\Delta$ belonging to a same minimal class, as well as those belonging to a class comparable with it, generate the same Cantor set (through the orbits of both lateral limits) and belong to it.

\subsection*{Acknowledgements}

AC was supported by the  PUCV Postgraduate Scholarship 2018. AC and PG were supported by Conicyt project REDES 180151, REDES REDI170457
 and FONDECYT 11714127. EC was partially financed by the Project of Research Group Nº 618 ``Sistemas Din\'amicos" of CSIC (Universidad de la Rep\'ublica, Uruguay) and by the Project  SNI 2015 2-1006119 of ANII (Uruguay).


\bibliographystyle{amsplain}

\begin{thebibliography}{9}

\bibitem{Br06} Br\'emont, J. Dynamics of injective quasi-contractions. {\it Ergodic theory and dynamical systems} {\bf 26}  (2006) 19-44.


\bibitem{B93}  Bugeaud, Y. Dynamique de certaines applications contractantes, lin\'eaires par morceaux, sur $[0,1[$. {\it Comptes rendus de l'Acad\'emie des sciences} S\'erie 1, Math\'ematique {\bf 317}  (1993) 575-578.

\bibitem{BC99} Bugeaud, Y. and Conze, J.-P. Calcul de la dynamique de transformations lin\'eaires contractantes mod 1 et arbre de Farey. {\it Acta Arithmetica} {\bf 88} (1999) 201-218.


\bibitem{CGM17} Catsigeras, E., Guiraud, P. and Meyroneinc, A. Complexity of injective piecewise contracting interval maps. {\it Ergodic Theory and Dynamical Systems} (2018) 1-25.



\bibitem{CGMU14} Catsigeras, E., Guiraud, P., Meyroneinc, A. and Ugalde, E. On the asymptotic properties of piecewise contracting maps. {\em Dynamical Systems} {\bf 31}  (2016), 107-135.


\bibitem{C} Coutinho, R. Din\'amica simb\'olica linear, Ph.D. Thesis. {\it Instituto Superior T\'ecnico, Universidade T\'ecnica de Lisboa} (1999).


\bibitem{GT88} Gambaudo, J.M. and Tresser, C. On the dynamics of quasi-contractions. {\it Bulletin of the Brazilian Mathematical Society} {\bf 19}  (1988) 61-114.

\bibitem{LN18} Laurent, M. and Nogueira, A. Rotation number of contracted rotations. {\it Journal of Modern Dynamics} {\bf 12} (2018) 175-191.

\bibitem{MH40} Morse, M. and Hedlund, G.A. Symbolic dynamics II. Sturmian trajectories. {\it American Journal of Mathematics} {\bf 62} (1940), 1-42.

\bibitem{NP15} Nogueira, A. and Pires, B. Dynamics of piecewise contractions of the interval. {\it Ergodic Theory and Dynamical Systems} {\bf 35} (2015) 2198-2215.

\bibitem{NPR14} Nogueira, A., Pires, B. and Rosales, R.A. Asymptotically periodic piecewise contractions of the interval. {\it Nonlinearity} {\bf 27}  (2014) 1603-1610.

\bibitem{NPR16} Nogueira, A., Pires, B. and Rosales, R.A. Topological dynamics of piecewise $\lambda$-affine maps. {\it Ergodic Theory and Dynamical Systems} (2016) 1-18.

\bibitem{P18} Pires, B. Symbolic dynamics of piecewise contractions. {\it ArXiv preprint arXiv:1803.01226} (2018).

\end{thebibliography}

\end{document}